\documentclass[hidelinks,onefignum,onetabnum]{siamart220329}

\usepackage{lipsum}
\usepackage{amsfonts}
\usepackage{graphicx}
\usepackage{epstopdf}
\usepackage{algorithmic}
\usepackage{amsmath}	
\usepackage{amscd}
\usepackage{amsfonts}
\usepackage{amssymb}		
\usepackage{mathtools}	
\usepackage{mathrsfs}	
\usepackage[english]{babel}		
\usepackage[utf8]{inputenc}	
\usepackage{calc}
\usepackage{thumbpdf}			
\usepackage{latexsym}
\usepackage{xcolor}
\usepackage{enumerate}		
\usepackage{url}				
\usepackage{cite}	
\usepackage{graphics} 
\usepackage{epsfig} 
\usepackage{color}
\usepackage{multirow}
\usepackage{booktabs}
\usepackage{adjustbox}

\ifpdf
  \DeclareGraphicsExtensions{.eps,.pdf,.png,.jpg}
\else
  \DeclareGraphicsExtensions{.eps}
\fi

\def\eps{\varepsilon}
\def\trace{{\rm Tr\,}}
\def\covar{{\rm Cov\,}}
\newcommand{\norm}[1]{\left\lVert#1\right\rVert}
\newcommand{\RF}{K}
\newcommand{\args}{(\cdot)}
\newcommand{\R}{\mathbb{R}}
\newcommand{\N}{\mathbb{N}}

\newcommand{\K}{\mathcal{K}}

\newcommand{\F}{\mathcal{F}}
\newcommand{\Ltwo}[1]{L^2(\Omega,\F,\mathbb{P};#1)}
\newcommand{\Lp}[2][p]{L^{#1}(\Omega,\F_k,\mathbb{P};#2)}

\newcommand{\Pset}{\mathcal{S}}
\newcommand{\Dset}{\mathcal{D}}
\newcommand{\Mset}{\mathcal{M}}
\newcommand{\Lset}{\mathcal{L}}
\newcommand{\Normal}{\mathcal{N}}
\newcommand{\Filtr}{(\F_k)_{k \in \N_0}}
\newcommand{\Prob}{\mathbb{P}}
\newcommand{\E}{\mathbb{E}}
\newcommand{\V}{\mathbb{V}}
\newcommand{\Exp}[1]{\mathbb{E}\left[#1\right]}

\newcommand{\Ell}{\ell}  
\newcommand{\J}{J} 

\newcommand{\xeq}{x^s}
\newcommand{\ueq}{u^s}
\newcommand{\Xstat}{X^s}
\newcommand{\Ustat}{U^s}
\newcommand{\XstatOpt}{X^s_{\star}}
\newcommand{\UstatOpt}{U^s_{\star}}
\newcommand{\XstatOptbar}{\bar{X}^s_{\star}}
\newcommand{\UstatOptbar}{\bar{U}^s_{\star}}

\newcommand{\DistrStatOptX}{\varrho^{s,\star}_{X}}

\newcommand{\Xshift}{\hat{X}}
\newcommand{\Ushift}{\hat{U}}
\newcommand{\Wshift}{\hat{W}}

\newcommand{\sshift}{\hat{r}}
\newcommand{\vshift}{\hat{v}}
\newcommand{\cshift}{\hat{c}}

\newcommand{\xeqopt}{\tilde{x}^s_{\star}}
\newcommand{\ueqopt}{\tilde{u}^s_{\star}}
\newcommand{\Xdist}{\tilde{X}}
\newcommand{\Udist}{\tilde{U}}

\newsiamremark{remark}{Remark}
\newsiamremark{example}{Example}

\headers{Turnpike and dissipativity in stochastic LQ OCPs}{J. Schießl, R. Ou, T. Faulwasser, M. H. Baumann, and L. Grüne}

\title{Turnpike and dissipativity in generalized discrete-time stochastic linear-quadratic optimal control
  \thanks{Submitted to the editors September 11, 2023. 
    \funding{This work was funded by the Deutsche Forschungsgemeinschaft (DFG, German Research Foundation) \-- project number 499435839.}
  }
}

\author{Jonas Schießl\thanks{Mathematical Institute, University of Bayreuth, Germany,
  (\email{jonas.schiessl@uni-bayreuth.de}, \email{michael.baumann@uni-bayreuth.de}, \email{lars.gruene@uni-bayreuth.de}).}
\and Ruchuan Ou\thanks{Institute of Control Systems, Hamburg University of Technology, Hamburg, Germany, main part of this work have been conducted while the authors were with the Institute of Energy Systems, Energy Efficiency and Energy Economics, TU Dortmund University,
Germany (\email{ruchuan.ou@tuhh.de}, \email{timm.faulwasser@ieee.org}).}
\and Timm Faulwasser \footnotemark[3]
\and Michael H. Baumann\footnotemark[2]
\and Lars Grüne\footnotemark[2]
}

\usepackage{amsopn}

\ifpdf
\hypersetup{
  pdftitle={Turnpike and dissipativity in generalized discrete-time stochastic linear-quadratic optimal control},
  pdfauthor={J. Schiessl, R. Ou, T. Faulwasser, M. H. Baumann, L. Gruene}
}
\fi

\begin{document}

\maketitle

\begin{abstract}
  We investigate different turnpike phenomena of generalized discrete-time stochastic linear-quadratic optimal control problems. Our analysis is based on a novel strict dissipativity notion for such problems, in which a stationary stochastic process replaces the optimal steady state of the deterministic setting. We show that from this time-varying dissipativity notion, we can conclude turnpike behaviors concerning different objects like distributions, moments, or sample paths of the stochastic system and that the distributions of the stationary pair can be characterized by a stationary optimization problem. The analytical findings are illustrated by numerical simulations.
\end{abstract}



\section{Introduction}

    In deterministic settings, the most common variant of turnpike behavior describes the phenomenon that optimal and near-optimal trajectories spend most of their time close to the optimal steady state independent of the time horizon of the problem. Thus, turnpike properties are valuable for analyzing the long-time behavior of (near-)optimal solutions of an optimal control problem (OCP). This has become particularly evident in the analysis of model predictive and receding horizon control schemes \cite{Faulwasser2018,Faulwasser2022,Gruene2013,Gruene2022}.
    Although the turnpike property was observed already in the first half of the 20th century by Ramsey~\cite{Ramsey1928} and von Neumann~\cite{Neumann1945} and its first in-depth theoretical study goes back to the 1950s \cite{Dorfman1958}, it is still subject of recent research.
    For deterministic problems, it is also well known that there is a strong relationship between the turnpike property and the concept of strict dissipativity introduced by Willems \cite{willems1972a,willems1972b}. This relationship was, for instance, investigated in \cite{Gruene2018} for linear-quadratic problems and in \cite{Faulwasser2017,Gruene2016} for general nonlinear systems. A remarkable result in this context is that strict dissipativity and the occurence of the turnpike behavior are equivalent under suitable conditions; see \cite{Gruene2016,Gruene2018}.
    
    While all these results apply to deterministic settings and the concepts of strict dissipativity and turnpike behavior are rather well understood there, much fewer results are known for stochastic systems. Whereas in the deterministic case it is clear how to characterize the distance between two solutions (at least for finite-dimensional problems), this is much more challenging for stochastic processes since there exist various metrics involving random variables and their distributions. So far, the literature was mainly concerned with analyzing turnpike behaviors in the sense of distributions or moments; see \cite{Kolokoltsov2012, Marimon1989, Sun2022}, and also existing dissipativity notions are defined on the underlying measure spaces; see \cite{Gros2022}. 
    In particular, \cite{Sun2022} establishes an exponential turnpike property for the expectation values of stochastic linear-quadratic OCPs with multiplicative noise in continuous time.
    Yet, numerical simulations for stochastic linear-quadratic OCPs also suggest turnpike properties for single realization paths; see \cite{Ou2021}, even in the case of additive noise, which always prevents the turnpike to be an equilibrium.
    First theoretical results on this pathwise behavior were presented in \cite{Sun2023} in a continuous-time setting with state and control-dependent noise, based on algebraic Riccati equations and backward stochastic differential equations. 
    In contrast to this, in this paper we present a dissipativity-based approach to analyze and characterize stochastic turnpikes, including pathwise phenomena.
    For this purpose, we introduce a new time-varying dissipativity notion for random variables considered as $L^2$ functions on the probability space. 
    This notion is based on the $L^2$ norm or mean-square distance of the solutions, where the deterministic steady state is replaced by a pair of stochastic processes with stationary distributions. 
    This concept of stochastic dissipativity differs from the---to the best of the authors knowledge---only alternative stochastic dissipativity definition tailored to optimal control problems in \cite{Gros2022} because the mean-square distance cannot be expressed by means of probability measures, which is at the core of the notion from \cite{Gros2022}. This allows us to conclude a larger set of turnpike properties from our new notion, including the already mentioned pathwise turnpike property in probability.
    Preliminary results in this direction are contained in the authors' recent conference paper \cite{CDCPaper}. The main contributions of the present paper compared to this conference paper are to show that the $L^2$ dissipativity allows us to directly infer various turnpike properties and to precisely characterize the corresponding turnpike object. More precisely, we show that our dissipativity notion does not only imply different turnpike properties for the pathwise behavior of the optimal solutions but also for the distributions and certain moments. 
    In our case, the distance between distributions is measured by the Wasserstein metric, as this is the natural counterpart to the $L^2$ norm we use on the space of random variables. 
    Moreover, we formulate a stationary optimization problem that uniquely characterizes the distribution of the stationary state process at which the turnpike phenomenon occurs. This constitutes another novelty of this paper and shows that this process is an appropriate generalization of the optimal steady state in the deterministic setting. It turns out that this optimization problem is the same by which the steady state distribution in \cite{Gros2022} is defined. Thus, it shows that we look at the same quantities when analyzing the distributions of optimal solutions. A further technical contribution is that we extend the results in \cite{CDCPaper} to stage costs containing linear terms and to arbitrary non-Gaussian noise with finite first and second moments.

  The remainder of this paper is organized as follows. 
  Section~\ref{sec:Setting} introduces the considered problem formulation and recalls the concepts of turnpike and dissipativity in the deterministic setting. 
  In Section~\ref{sec:Dissi}, we introduce the pair of stochastic stationary processes that replaces the deterministic steady state and define a time-varying notion of stochastic dissipativity in $L^2$. Further, we explicitly construct a storage function for the generalized stochastic linear-quadratic OCP and, thus, prove its strict dissipativity. 
  Section~\ref{sec:Turnpike} then shows that this dissipativity notion is strong enough to conclude turnpike properties for the individual realization paths in probability, as well as for distributions and their moments. 
  Moreover, we show in Section~\ref{sec:StatOpt} that the distribution of the stationary state process is uniquely characterized by the solution of a stationary optimization problem, and we illustrate our results by numerical examples in Section~\ref{sec:Examples}. 
  Section~\ref{sec:Conclusion} concludes the paper.

  \textbf{Notation:} With $\N$, $\R$ we denote the natural and real numbers and we set $\N_0 := \N \cup \{0\}$. For a quadratic matrix $Q \in \R^{n \times n}$ we write $Q \geq 0$ if it is positive semidefinite and $Q > 0$ if it is positive definite. $\trace(Q)$ denotes the trace of a matrix $Q \in \R^{n \times n}$ and for $Q \geq 0$ and $x \in \R^n$ we abbreviate $\Vert x \Vert_Q := \sqrt{x^T Q x}$. If $Q = I$ is the identity we write $\Vert x \Vert = \Vert x \Vert_I$. For a random variable $X$ the expected value and covariance are denoted as $\E[X]$ and $\covar(X)$ and if $X$ is onedimensional $\V[X] := \covar(X)$ is the variance of $X$.

\section{Setting and preliminary results} \label{sec:Setting}

  \subsection{Problem formulation} 
  For $A \in \R^{n \times n}$, $B \in \R^{n \times l}$, with $(A,B)$ stabilizable, $E \in \R^{n \times m}$, and $z \in \R^n$ we consider linear stochastic systems of the form 
  \begin{equation} \label{eq:StochSystem}
      X(k+1) = AX(k) + BU(k) + EW(k) + z, \quad X(0) = X_0
  \end{equation}
  where for each $k \in \N_0$,  $X(k) \in \Lp[2]{\R^n}$,  $U(k) \in \Lp[2]{\R^l}$ and $W(k) \in \Ltwo{\R^m}$ holds. Here $(\Omega,\F,\Prob)$ is a probability space, $\Omega$ is the set of realizations, $\Prob$ is the probability measure, $\F$ is a $\sigma$-algebra, and $\Filtr$ is a filtration following the usual hypotheses of \cite{Protter2005}.
  The value $z \in \R^n$ represents a non-stochastic additive constant which could appear in system \eqref{eq:StochSystem}. 
  Similar models as \eqref{eq:StochSystem} were also considered in \cite{Aastroem1970,Caines2018,Kumar2015} in a general stochastic control theoretic framework and utilized in \cite{Mesbah2016,Heirung2018,Hewing2020,Farina2015} to formulate stochastic predictive control schemes.
  For our purpose, we choose $\Filtr$ as the smallest filtration such that $X$ is an adapted process, i.e.
  \begin{equation} \label{eq:Filtration}
    \F_k = \sigma(X(0),\ldots,X(k)), ~ \text{for all} ~ k \in \N_0.
  \end{equation}
  This choice of the stochastic filtration is also called the minimal \cite{Fristedt1997} or natural filtration \cite{Protter2005} of the stochastic process $X\args$ and induces a causality requirement, which guarantees that we only take past but not future events into account for our control design.
  More precisely, the filtration \eqref{eq:Filtration} ensures that the control action $U(k)$ at time $k$ only depends on the observed states $\lbrace X(s) \rbrace_{s=0,\ldots,k-1}$ and the influence of the disturbances $\lbrace W(s) \rbrace_{s=0,\ldots,k-1}$ is implicitly handled via the state recursion. 
  Thus, our control laws correspond to state feedbacks and not to direct disturbance feedbacks.
  We refer to \cite{Fristedt1997, Protter2005} for more details on stochastic filtrations.\\
  Further, we assume that $\lbrace W(k) \rbrace_{k \in \N_0}$ is an arbitrary but fixed sequence of \textit{i.i.d.} random variables and that for every $k \in \N_0$ the random variable $W(k) \sim \varrho_W$ is independent of $X(k)$ and $U(k)$. Here $\varrho_W$ is a not necessarily Gaussian probability distribution with finite mean $\mu_W \in \R^m$ and covariance matrix $\Sigma_W \in \R^{m \times m}$.
  For a given initial value $X_0 \sim \varrho_0$ and control $U\args$, we denote the solution of system \eqref{eq:StochSystem} by $X_{U}(\cdot,X_0)$, or short by $X\args$ if the initial value and the control are unambiguous. Note, that the solution $X_{U}(\cdot,X_0)$ also depends on the disturbance $W(\cdot)$. However, for the sake of readability, we do not highlight this in our notation. \\
  We abbreviate the dynamics
  \begin{equation}
      X(k+1) = f(X(k),U(k),W(k)) := AX(k) + BU(k) + EW(k) + z
  \end{equation}
  and define the stage cost $\Ell: \Ltwo{\R^n} \times \Ltwo{\R^l} \rightarrow \R$ as
  \begin{equation} \label{eq:stochStagecost}
      \begin{split}
          \Ell(X,U) :=& \Exp{X^T Q_1 X + U^T R_1 U + r^T X + v^T U + c} \\
          &+ \trace \big( Q_2 \covar(X) + R_2 \covar(U) \big),
      \end{split}
  \end{equation}
  where $Q_1$, $Q_2 \in \R^{n \times n}$, and $R_1$, $R_2 \in \R^{l \times l}$, with $Q_1$, $Q_2$, $R_1$ and $R_2$ symmetric, $Q_1,Q_2,R_2 \geq 0$ and $R_1 > 0$ and $s \in \R^n$, $v \in \R^l$, $c \in \R$.

  \begin{remark}
    Note, since 
    \begin{equation} \label{eq:SecondMoment}
        \Exp{X^T Q X} = \trace (Q \covar (X)) + \Exp{X}^T Q \Exp{X}
    \end{equation}
    holds, the quadratic terms in the first line of the stage cost \eqref{eq:stochStagecost} always jointly penalize the quadratic mean and the variance of $X$ and $U$, whereas the second part of the stage cost allows us to model an additional variance penalization that has no impact on the mean.
    While this gives us an additional degree of freedom in the design of the cost function, it is a priori not clear if the mean-square $\Exp{X}^T Q \Exp{X}$ term appearing in \eqref{eq:SecondMoment} has an impact on the solvability of the optimization problem \eqref{eq:sOCP}. However, since $\trace (Q \covar (X,Y))$ is a symmetric bilinear form for all $Q \geq 0$ and $\Exp{X^T R Y}$ is an inner product on $\Ltwo{\R^n}$ for all $R > 0$, we can conclude that $\Exp{X^T Q_1 X} + \trace ( Q_2 \covar(X) )$ from \eqref{eq:stochStagecost} is positive semidefinite and $\Exp{U^T R_1 U} + \trace ( R_2 \covar(U) )$ is positive definite, which guarantees the solvability of the problem.
  \end{remark}
  
  Then the generalized discrete-time stochastic linear-quadratic optimal control problem (OCP) under consideration reads
  \begin{equation} \label{eq:sOCP}
      \begin{split}
          \min_{U\args} \J_N(X_0,U) &:= \sum_{k=0}^{N-1} \Ell(X(k),U(k)) \\
          s.t. ~ X(k+1) &= f(X(k),U(k),W(k)), \quad X(0) = X_0.
      \end{split}
  \end{equation}

  \begin{remark}
      Additional terminal costs may be considered in \eqref{eq:sOCP}. However, since dissipativity is a property of the stage cost and the dynamics but not of the terminal cost, cf.\ Definition~\ref{defn:StochDissi}, we omit this generalization for the discussion in this paper.
  \end{remark}

  \subsection{Deterministic Dissipativity and Turnpike Properties} \label{sec:DetResults}

  Before investigating stochastic properties, we recall the basic concepts of dissipativity and turnpike in the deterministic case.
  Let us denote the deterministic counterpart to system \eqref{eq:StochSystem} by
  \begin{equation} \label{eq:DetSystem}
      x(k+1) = f(x(k),u(k),\mu_W) := Ax(k) + Bu(k) + E \mu_W + z, \quad x(0) = x_0
  \end{equation}
  where $x(k) \in \R^n$, $u(k) \in \R^l$, and $x_0 \in \R^n$. 
  Then the deterministic stage cost is given by
  \begin{equation}
      \ell(x,u) := x^T Q_1 x + u^T R_1 u + r^T x + v^T u + c
  \end{equation}
  for all $(x,u) \in \R^n \times \R^l$ and the deterministic version of problem~\eqref{eq:sOCP} reads
  \begin{equation} \label{eq:OCP}
      \begin{split}
          \min_{u\args} J_N(x_0,u) &:= \sum_{k=0}^{N-1} \ell(x(k),u(k)) \\
          s.t. ~ x(k+1) &= f(x(k),u(k),\mu_W), \quad x(0) = x_0.
      \end{split}
  \end{equation}
  Next, we want to define strict dissipativity in the deterministic case. 
  For this purpose we recall that a pair $(\xeq,\ueq) \in \R^n \times \R^l$ is called steady state of the deterministic system if $\xeq = f(\xeq,\ueq,\mu_W)$ and define the class of comparison functions
  \begin{equation*}
      \begin{split}
          \K_{\infty} := \lbrace \alpha: \R_0^+ \rightarrow \R_0^+ ~\mid~ & \alpha \text{ is continuous, strictly increasing \& unbounded} \\
          & \text{with } \alpha(0)=0 \rbrace.
      \end{split}
  \end{equation*}
  
  \begin{definition}[Deterministic Strict Dissipativity]
      Given a steady state $(\xeq,\ueq)$ the deterministic optimal control problem \eqref{eq:OCP} is called \textit{strictly dissipative} if there exists a storage function $\lambda: \R^n \rightarrow \R$ bounded from below and a function $\alpha \in \K_{\infty}$ such that 
      \begin{equation}
          \ell(x,u) - \ell(\xeq,\ueq) + \lambda(x) - \lambda(f(x,u,\mu_W)) \geq \alpha(\Vert x - \xeq \Vert)
      \end{equation}
      holds for all $(x,u) \in \R^n \times \R^l$.
  \end{definition}

  Another tool for analyzing optimal control problems is the so-called turnpike property. This property exists in various forms; see, e.g., the overview in \cite{Faulwasser2022}. The deterministic variant that is most directly linked to dissipativity is formalized in the following definition.

  \begin{definition}[Deterministic Turnpike] \label{defn:Turnpike}
      Given $(\xeq,\ueq)$ the deterministic optimal control problem \eqref{eq:OCP} has the turnpike property if there exists $\alpha \in \K_{\infty}$ such that the following holds:
      For every $x_0 \in \R^n$ there exist $C > 0$ such that for each $N \in \N$, each control $u\args$ satisfying $J_N(x,u) \leq N \ell(\xeq,\ueq) + \delta$ and each $\eps > 0$ the value $Q_{\eps} := \# \lbrace k \in \lbrace 0,\ldots,N-1 \rbrace \mid \Vert x_u(k,x_0)-\xeq \Vert \leq \eps \rbrace$ satisfies the inequality $Q_{\eps} \geq N - (\delta + C)/\alpha(\eps)$.
  \end{definition}

  The turnpike property defined as in the above theorem says that for any neighborhood of the optimal steady state, for increasing horizon lengths $N$ the optimal trajectories spend most of their time in this neighborhood. More precisely, the number of time instances at which the trajectories are outside this neighborhood is bounded by a constant independent of the horizon length $N$.

  The connection between strict dissipativity and turnpike properties for deterministic OCPs is already well studied; see \cite{Gruene2016,Gruene2018,Gruene2022}. 
  In particular one can prove that strict dissipativity implies turnpike behavior for the deterministic case.

  \begin{lemma}[\hspace{-0.1pt}{\cite[Theorem 5.3]{Gruene2013}}] \label{lem:DetTurnpike}
      Assume that the optimal control problem \eqref{eq:OCP} is strictly dissipative at $(\xeq,\ueq)$. 
      Then it also has the turnpike property in the sense of Definition~\ref{defn:Turnpike}.    
  \end{lemma}

\section{Stationary solutions and stochastic dissipativity} \label{sec:Dissi}

  Our goal in the remainder of this paper is to generalize the deterministic concepts of turnpike and dissipativity from Section~\ref{sec:DetResults} to stochastic optimal control problems.
  For that, we must first clarify how to define a stochastic version of a deterministic steady state.
  We note that, for any $U \in \Lp[2]{\R^l}$, the condition
  \begin{equation}
      X = AX + BU + EW
  \end{equation}
  cannot be satisfied as $U$ may not depend on $W$ according to the underlying filtration. Thus, a random variable that is constant in time is not a suitable candidate for the stochastic counterpart of the optimal steady state. 
  However, we can keep the distribution of such a process constant instead and use the following definition of stationary processes. 

  \begin{definition} \label{defn:StationaryPair}
    Let $\{W(k)\}_{k \in \N_0}$ be the same sequence of random variables as in \eqref{eq:StochSystem}. 
    Then a pair of state and control processes $(\Xstat(\cdot),\Ustat(\cdot))$ is called stationary for system~\eqref{eq:StochSystem} if there are stationary distributions $\varrho_X^s \in \mathcal{P}_2(\R^n)$ and $\varrho_U^s \in \mathcal{P}_2(\R^l)$ such that
    \begin{equation*}
      \Xstat(k) \sim \varrho_X^s,~\Ustat(k) \sim \varrho_U^s, ~ \mbox{and} ~ \Xstat(k+1) = f(\Xstat(k),\Ustat(k),W(k))
    \end{equation*}
    for all $k \in \N_0$. Here $\mathcal{P}_2(\mathcal{X}) := \{ \varrho \in \mathcal{P}(\mathcal{X}) \mid \int_{\mathcal{X}} \Vert x \Vert^2 \varrho(dx) < \infty \}$ and $\mathcal{P}(\mathcal{X})$ denotes the collection of all probability measures on the measurable space $\mathcal{X}$.
  \end{definition}

  \begin{remark} \label{rem:StatProbMeasure}
    The probability measure defining a stationary distribution is called an invariant measure; see \cite{Meyn1993}.
    Hence, an alternative approach to the above stationarity concept could be to switch entirely to the set of underlying probability measures and to conduct the analysis there as in \cite{Gros2022}.
    However, this approach is limited because the metric used to measure the distances only depends on the distribution or the moments of the random variables and not on their representation as a $L^2$ function. Thus, all information about the single realization paths of the solutions is lost---although numerical simulations suggest a turnpike phenomenon for the paths, as well; see \cite{Ou2021}.
  \end{remark}

  \begin{remark} \label{rem:FeedbackStationarity}
    Let $\mathbb{F}(\R^n,\R^l)$ denote the space of measurable maps $\pi:\R^n \to\R^l$ for which $\pi\circ X \in \Lp[2]{\R^l}$ for each $X\in\Lp[2]{\R^n}$.
    If $\Ustat\args$ can be written in a feedback form, i.e., $\Ustat\args = \pi(X^s\args)$ for some $\pi\in \mathbb{F}(\R^n,\R^l)$, then the additional condition $\Ustat(k) \sim \varrho^s_U \in \mathcal{P}_2(\R^l)$ can be omitted since it is always fulfilled if the distribution $\varrho_X^s \in \mathcal{P}_2(\R^n)$ of $\Xstat\args$ is stationary. 
  \end{remark}

  Using Definition \ref{defn:StationaryPair} to characterize stochastic stationarity, we can give the following stochastic version of strict dissipativity for random variables, which, as we will see later, is strong enough to conclude a turnpike behavior of single realization paths.

  \begin{definition}[Mean-square Dissipativity] \label{defn:StochDissi}
    Let $(\Xstat\args,\Ustat\args)$ be a stationary pair according to Definition \ref{defn:StationaryPair}. Then the stochastic OCP \eqref{eq:sOCP} is called \emph{strictly dissipative in $L^2$} (or in \emph{mean square}) at $(\Xstat\args,\Ustat\args)$ if there exists a time-varying storage function $\lambda: \N_0 \times \Ltwo{\R^n}\rightarrow \R$ uniformly bounded from below in the second argument and a function $\alpha \in \K_{\infty}$ such that 
    \begin{equation} \label{eq:sDI}
        \begin{split}
            &\Ell(X(k),U(k)) - \Ell(X^s(k),U^s(k)) + \lambda(k,X(k)) - \lambda(k+1,f(X(k),U(k),W(k))) \\
            &\quad \geq \alpha \Big( \norm{ X(k) - \Xstat(k) }_{L^2}^2 \Big) = \alpha \Big( \E \big[ \norm{ X(k) - \Xstat(k) }^2 \big] \Big)
        \end{split}
    \end{equation}
    holds for all $k \in \N_0$ and $(X(k),U(k)) \in \Lp[2]{\R^n} \times \Lp[2]{\R^l}$.
  \end{definition}

  Dissipation inequalities like \eqref{eq:sDI} have originally been introduced to model energy dissipation in open physical systems. In the context of optimal control, this inequality constitutes an abstract relation between the stage cost and the storage function. From an optimal control point of view, it is notable that the difference of the storage functions on the left side of \eqref{eq:sDI} gives a lower bound on the optimal value function. 

  \begin{remark}
    We note that all terms in this inequality depend on the realizations of the random variables $X^s(k)$, $U^s(k)$ and, implicitly via $f$, on $W(k)$. In contrast to the respective distributions of these random variables, these quantities are time-varying elements of the corresponding $L^2$ spaces. Thus, inequality \eqref{eq:sDI} cannot be satisfied with a storage function constant in time and the dissipativity notion becomes time-varying just as in the deterministic case when the setting becomes time-varying \cite{Gruene2018a,Risbeck2020,Mueller2021}.
    Further, in the resulting dissipativity inequality, the distance measure on the right-hand side also depends on the exact realization of the random variables rather than only on their distributions, cf.\ \eqref{eq:sDI}.
    That means our distance measure is not a dissimilarity measure in the sense of \cite{Gros2022}, and, thus, our concept of stochastic dissipativity differs from the one presented there.
    Moreover, as we will see in Section~\ref{sec:Turnpike}, our dissipativity notion implies both turnpike results regarding the distributions as it would also be possible with the notion of \cite{Gros2022} and, in addition, pathwise turnpike properties. Thus, our dissipativity notion is strong enough to conclude a larger set of turnpike properties as the one presented in \cite{Gros2022}.
  \end{remark}

  The remainder of this section will show how we can construct a storage function $\lambda$ for OCP \eqref{eq:sOCP} to show strict dissipativity in the sense of Definition \ref{defn:StochDissi}. This construction is based on reformulations of the stage cost divided into several steps in a series of lemmas, starting with Lemma \ref{lem:PushShiftToCost}. 

  \begin{lemma} \label{lem:PushShiftToCost}
    Assume that the equation $(I-A)\xeq - B\ueq = E\mu_W + z$
    has a solution $(\xeq,\ueq)$, i.e., the deterministic system \eqref{eq:DetSystem} has a steady state. 
    Then the stochastic OCP \eqref{eq:sOCP} can be reformulated as
    \begin{align*}
      \min_{\Ushift\args} \J_N(X_0,\Ushift) := \sum_{k=0}^{N-1}&\Exp{\Xshift(k)^T Q_1 \Xshift(k) + \Ushift(k)^T R_1 \Ushift(k) + \sshift^T \Xshift(k)+ \vshift^T \Ushift(k) + \cshift} \\
      &+ \trace \big( Q_2 \covar(\Xshift(k)) + R_2 \covar(\Ushift(k)) \big) \\
      s.t. ~ \Xshift(k+1) &= f(\Xshift(k),\Ushift(k),\Wshift(k)) - z, \quad \Xshift(0) = X_0 - \xeq,
    \end{align*}
    where $\Xshift\args := X\args - \xeq$, $\Ushift\args := U\args - \ueq$, $\sshift := r + 2Q_1\xeq$, $\vshift := v + 2R_1\ueq$, $\cshift := {\xeq}^T Q_1 \xeq + {\ueq}^T R_1 \ueq + r^T \xeq + v^T \ueq + c$, and $\Wshift\args := W\args - \mu_W$.
  \end{lemma}
  \begin{proof}
    Since $\Xshift(0) = X(0) - \xeq$ by the definition of the initial condition, we can compute that for every $X(k) \in \Lp[2]{\R^n}$ and $U(k) \in \Lp[2]{\R^l}$ it holds that 
    \begin{equation*}
      \begin{split}
         f(\Xshift(k),\Ushift(k),\Wshift(k)) &-z = A\Xshift(k) + B\Ushift(k) + E\Wshift(k) +z -z\\
        &= A(\Xshift(k) - \xeq) + B(\Ushift(k) - \ueq) + E(W(k)-\mu_W) + z - z \\
        &= (AX(k) + BU(k) + EW(k) + z) - (A\xeq  +B\ueq +E\mu_W +z) \\
        &= X(k+1) - \xeq = \Xshift(k+1) .
      \end{split}
    \end{equation*}
    Thus, the original stage costs in terms of the new variables $(\Xshift\args,\Ushift\args)$ are given by 
    \begin{equation*}
      \begin{split}
        \ell(X(k),U(k)) =& \Exp{(\Xshift(k) + \xeq)^T Q_1 (\Xshift(k) + \xeq) + (\Ushift(k) + \ueq)^T R_1 (\Ushift(k) + \ueq)} \\ &+ \Exp{r^T (\Xshift(k) + \xeq) + v^T (\Ushift(k) + \ueq) + c} \\
        &+ \trace \big( Q_2 \covar(\Xshift(k) + \xeq) + R_2 \covar(\Ushift(k) + \ueq) \big) \\
        =& \Exp{\Xshift(k)^T Q_1 \Xshift(k) + \Ushift(k)^T R_1 \Ushift(k) + \sshift^T \Xshift(k)+ \vshift^T \Ushift(k) + \cshift} \\
        &+ \trace \big( Q_2 \covar(\Xshift(k)) + R_2 \covar(\Ushift(k)) \big),
      \end{split}
    \end{equation*} 
    which proves the claim.
  \end{proof}

  \begin{remark}
      Note that the assumption of the existence of a solution to the equation $(I-A)\xeq - B\ueq = E\mu_w + z$ in Lemma~\ref{lem:ReformulationStagecost} is needed, since there may be certain choices of $\mu_w \in \R$ and $z \in \R$ for which there is no solution if we only assume $(A,B)$ to be stabilizable. However, if instead we assume the stronger condition $(A,B)$ controllable, then this additional assumption can be omitted.
  \end{remark}

  The contribution of Lemma \ref{lem:PushShiftToCost} is that we have moved the constant terms $E \mu_W + z$ from the system dynamics to the stage cost
  \begin{equation} \label{eq:ShiftedStagecosts}
    \begin{split}
      \hat{\ell}(\Xshift,\Ushift) :=& \Exp{\Xshift^T Q_1 \Xshift + \Ushift^T R_1 \Ushift + \sshift^T \Xshift+ \vshift^T \Ushift + \cshift} \\
      &+ \trace \big( Q_2 \covar(\Xshift) + R_2 \covar(\Ushift) \big).
    \end{split}
  \end{equation}
  Thus, we only have to handle the constant terms in the stage costs while the system now reads 
  \begin{equation} \label{eq:StochSystemZeroMean}
    \Xshift(k+1) = A \Xshift(k) + B \Ushift(k) + E \Wshift(k),
  \end{equation}
  where $\Wshift(k)$ has zero mean for all $k \in \N_0$.
  For this reason, we will first construct a storage function for the stochastic OCP with stage costs $\hat{\ell}(\Xshift,\Ushift)$ from \eqref{eq:ShiftedStagecosts} and state process $\Xshift\args$ from system \eqref{eq:StochSystemZeroMean}. Then, in Theorem \ref{thm:StochDissi}, we will use Lemma \ref{lem:PushShiftToCost} to conclude that the original problem is strictly dissipative in the $L^2$ sense.
  The next lemma gives us a reformulation of the stage costs $\hat{\ell}(\Xshift,\Ushift)$ especially concerning the quadratic parts of the stage costs.

  \begin{lemma} \label{lem:ReformulationStagecost}
    Let $P$ be the positive semidefinite solution of the discrete-time algebraic Riccati equation 
    \begin{equation} \label{eq:algRiccati}
      P = A^T P A + (Q_1 + Q_2) - A^T P B [ (R_1 + R_2) + B^T P B ]^{-1} B^T P A,
    \end{equation}
    and set $\RF := -\left[(R_1 + R_2) + B^T P B\right]^{-1} B^T P A$.
    Then for every $k \in \N_0$ and any pair $(X(k),U(k)) \in \Lp[2]{\R^n} \times \Lp[2]{\R^l}$ the identity
    \begin{equation}
        \begin{split}
            \hat{\ell}(\Xshift(k),\Ushift(k)) = & \Exp{\Vert \Ushift(k) - \RF \Xshift(k) \Vert_{\tilde{R}}^2} + \Exp{\Vert E\Wshift(k) \Vert_P^2} + \Exp{\Vert \Xshift(k) \Vert_P^2} \\
            &- \Exp{\Vert \Xshift(k+1) \Vert_P^2} + \Exp{\sshift^T \Xshift(k)+ \vshift^T \Ushift(k) + \cshift} \\
            & - \Exp{\Xshift(k)}^T Q_2 \Exp{\Xshift(k)} - \Exp{\Ushift(k)}^T R_2 \Exp{\Ushift(k)}
        \end{split}
    \label{eq:lhat}
    \end{equation}
    holds with $\tilde{R} := (R_1 + R_2) + B^TPB$ symmetric and positive definite.
  \end{lemma}
  \begin{proof}
    Using equation \eqref{eq:SecondMoment} we can rewrite the stage costs \eqref{eq:stochStagecost} as 
    \begin{equation*}
      \begin{split}
        \hat{\ell}(\Xshift(k),\Ushift(k)) =& \Exp{\Xshift(k)^T Q_1 \Xshift(k) + \Ushift(k)^T R_1 \Ushift(k) + \sshift^T \Xshift(k)+ \vshift^T \Ushift(k) + \cshift} \\
        &+ \Exp{\Xshift(k)^T Q_2 \Xshift(k)} - \Exp{\Xshift(k)}^T Q_2 \Exp{\Xshift(k)} \\
        &+ \Exp{\Ushift(k)^T R_2 \Ushift(k)} - \Exp{\Ushift(k)}^T R_2 \Exp{\Ushift(k)} \\
        =& \Exp{\Xshift(k)^T (Q_1+Q_2) \Xshift(k) + \Ushift(k)^T (R_1 + R_2) \Ushift(k)} \\ 
        &+ \Exp{\sshift^T \Xshift(k)+ \vshift^T \Ushift(k) + \cshift} - \Exp{\Xshift(k)}^T Q_2 \Exp{\Xshift(k)} \\
        & - \Exp{\Ushift(k)}^T R_2 \Exp{\Ushift(k)}
      \end{split}
    \end{equation*}
    for all $k \in \N_0$ and $(X(k),U(k)) \in \Lp[2]{\R^n} \times \Lp[2]{\R^l}$. Further, the dynamics of $\Xshift\args$ are given by system \eqref{eq:StochSystemZeroMean}, where $\Wshift(k)$ has zero mean and is independent of $\Xshift(k)$ and $\Ushift(k)$ for all $k \in \N_0$. 
    And since $P$ is the solution of the algebraic Riccati equation \eqref{eq:algRiccati} it is also a steady-state solution of the Riccati difference equation
    \begin{equation}
        \begin{split}
            P_N(k) =& A^T P_N(k+1) A + (Q_1 + Q_2) - A^T P_N(k+1)B \\ 
            &\times [ (R_1 + R_2) + B^T P_N(k+1) B ]^{-1} B^T P_N(k+1)A
        \end{split}
    \end{equation}
    with terminal condition $P_N(N) = P$. Hence by \cite[Chapter~8, Lemma~6.1]{Aastroem1970} we conclude that 
    \begin{equation} \label{eq:DissiStrategy}
      \begin{split}
        &\Exp{\Xshift(k)^T (Q_1+Q_2) \Xshift(k) + \Ushift(k)^T (R_1 + R_2) \Ushift(k)} \\
        &= \Exp{\Vert \Ushift(k) - \RF \Xshift(k) \Vert_{\tilde{R}}^2} + \Exp{\Vert E\Wshift(k) \Vert_P^2} \\
        &\quad + \Exp{\Vert \Xshift(k) \Vert_P^2} - \Exp{\Vert \Xshift(k+1) \Vert_P^2},
      \end{split}
    \end{equation}
    which proves the lemma.
  \end{proof}

  Equation \eqref{eq:lhat} gives us a form of the stage cost that is suitable for the construction of the storage function for our stochastic linear-quadratic problem, starting with the following lemma.

  \begin{lemma} \label{lem:StagecostShiftedArguments}
    Assume that $(A,Q^{1/2})$ is detectable. Then there exists an invariant distribution $\hat{\varrho}_X^s \in \mathcal{P}_2(\R^n)$ and an initial condition $\Xshift^s(0) = \Xshift^s_0 \sim \hat{\varrho}_X^s$ such that the state and control processes defined by
    \begin{align*}
      \Xshift^s(k+1) &= A \Xshift^s(k) + B \Ushift^s(k) + E\Wshift(k) \\
      \Ushift^s\args &= K \Xshift^s\args
    \end{align*}
    are a stationary pair according to Definition \ref{defn:StationaryPair} and $\E[\Xshift^s(k)] = 0$ holds for all $k \in \N_0$.
    Moreover, for every $k \in \N_0$ and $(X(k),U(k)) \in \Lp[2]{\R^n} \times \Lp[2]{\R^l}$, the identity 
    \begin{equation} \label{eq:FirstModifiCost}
        \begin{split}
            &\hat{\ell}(\Xshift(k),\Ushift(k)) - \hat{\ell}(\Xshift^s(k),\Ushift^s(k)) + \hat{\lambda}(k,\Xshift(k)) - \hat{\lambda}(k+1,\Xshift(k+1)) \\
            &= \hat{\ell}(\Xshift(k)-\Xshift^s(k),\Ushift(k)-\Ushift^s(k)) + \hat{c}
        \end{split}
    \end{equation}
    holds with $\hat{c}$ from Lemma~\ref{lem:ReformulationStagecost} and $\hat{\lambda}(k,\Xshift(k))$ defined as
    \begin{equation*}
        \hat{\lambda}(k,\Xshift(k)) = \Exp{\Vert \Xshift(k) - \Xshift^s(k) \Vert_P^2 - \Vert \Xshift(k) \Vert_P^2}. 
    \end{equation*}
  \end{lemma}
  \begin{proof}
    Since $(A,Q^{1/2})$ is detectable, $\RF$ is a stabilizing feedback and thus $A+B\RF$ is Schur-stable. Hence, since the first two moments of $\Wshift(k)$ are finite for all $k \in \N_0$, the stationarity of $\Xshift^s\args$ follows by \cite[Section~10.5.4]{Meyn1993} and the stationarity of $\Ushift^s\args$ by its feedback form, cf. Remark~\ref{rem:FeedbackStationarity}.
    Additionally, \cite[Section~10.5.4]{Meyn1993} ensures that $\Xshift^s(k) \in \Lp[2]{\R^n}$ holds for all $k \in \N_0$ which implies $\hat{\varrho}_X^s \in \mathcal{P}_2(\R^n)$.
    Further, since $\E[\Wshift(k)] = 0$, it follows directly that $\E[\Xshift^s(k)] = 0$ for all $k \in \N_0$ since it is the only finite solution of the steady state equation $\mu_{\Xshift}^s = (A+B\RF) \mu_{\Xshift}^s$ for the expectation $\mu_{\Xshift}^s$ of $\Xshift^s\args$.
    To conclude \eqref{eq:FirstModifiCost}, we first observe that 
        \begin{equation} \label{eq:invarianceQterm}
            \begin{split}
                \Exp{\Vert \Ushift(k) - \RF \Xshift(k) \Vert_{\widetilde{R}}^2} &= \Exp{\Vert \Ushift(k) -\RF \Xshift^s(k) + \RF \Xshift^s(k) - \RF \Xshift(k)  \Vert^2_{\widetilde{R}}} \\
                &= \Exp{\Vert \left( \Ushift(k) - \Ushift^s(k) \right) - \RF \left(\Xshift(k) - \Xshift^s(k) \right) \Vert_{\widetilde{R}}^2}
            \end{split}
        \end{equation}
        holds. Let us define
        \begin{equation*}
          \begin{split}
            \bar{\ell}(\Xshift(k),\Ushift(k)) :=& \hat{\ell}(\Xshift(k),\Ushift(k)) - \Exp{\Vert E\Wshift(k) \Vert_P^2} - \Exp{\Vert \Xshift(k) \Vert_P^2} \\
            &+ \Exp{\Vert A\Xshift(k) + B\Ushift(k) + E\Wshift(k) \Vert_P^2}.
          \end{split}
        \end{equation*}
        Then, since $\Xshift^s(k)$ has zero mean for all $k \in \N_0$, it follows by Lemma \ref{lem:ReformulationStagecost} and equation \eqref{eq:invarianceQterm} that 
        \begin{equation} \label{eq:TermDissi1}
          \begin{split}
            \bar{\ell}(\Xshift(k),\Ushift(k)) =& \bar{\ell}(\Xshift(k)-\Xshift^s(k),\Ushift(k)-\Ushift^s(k)) \\
            =& \hat{\ell}(\Xshift(k)-\Xshift^s(k),\Ushift(k)-\Ushift^s(k)) \\
            &- \Exp{\Vert E\Wshift(k) \Vert_P^2} - \Exp{\Vert \Xshift(k)-\Xshift^s(k)\Vert_P^2} \\
            &+ \Exp{\Vert A(\Xshift(k)-\Xshift^s(k)) + B(\Ushift(k)-\Ushift^s(k)) + E\Wshift(k) \Vert_P^2}.
          \end{split}
        \end{equation}
        Moreover, because $(\Xshift(k),\Ushift(k))$ as well as $(\Xshift^s(k),\Ushift^s(k))$ is stochastically independent of $\Wshift(k)$ we get 
        \begin{equation}
          \begin{split}
            &\Exp{\Vert A(\Xshift(k)-\Xshift^s(k)) + B(\Ushift(k)-\Ushift^s(k)) + E\Wshift(k) \Vert_P^2} \\
            &= \Exp{\Vert (A\Xshift(k) + B\Ushift(k))-(A\Xshift^s(k) + B\Ushift^s(k)) \Vert_P^2} + \Exp{\Vert E\Wshift(k) \Vert_P^2} \\
            &= \Exp{\Vert \Xshift(k+1)-\Xshift^s(k+1)\Vert_P^2} + \Exp{\Vert E\Wshift(k) \Vert_P^2}.
          \end{split}
        \end{equation}
        Thus, after resolving equation \eqref{eq:TermDissi1} we get 
        \begin{equation*}
          \begin{split}
            &\hat{\ell}(\Xshift(k),\Ushift(k)) - \Exp{\Vert E\Wshift(k) \Vert_P^2} + \Exp{\Vert \Xshift(k)-\Xshift^s(k)\Vert_P^2} \\
            &- \Exp{\Vert \Xshift(k) \Vert_P^2} - \Exp{\Vert \Xshift(k+1)-\Xshift^s(k+1)\Vert_P^2} + \Exp{\Vert \Xshift(k+1) \Vert_P^2} \\
            &= \hat{\ell}(\Xshift(k),\Ushift(k)) - \Exp{\Vert E\Wshift(k) \Vert_P^2} + \hat{\lambda}(k,\Xshift(k)) - \hat{\lambda}(k+1,\Xshift(k+1)) \\
            &= \hat{\ell}(\Xshift(k)-\Xshift^s(k),\Ushift(k)-\Ushift^s(k)).
          \end{split}
        \end{equation*} 
        Moreover, by evaluating equation \eqref{eq:lhat} at $(\Xshift^s(k),\Ushift^s(k))$ and using that $\Xshift^s(k)$ has zero mean, we obtain
        \begin{equation*}
            \begin{split}
                \hat{\ell}(\Xshift^s(k),\Ushift^s(k)) =& \Exp{\Vert \Ushift^s(k) - \RF \Xshift^s(k) \Vert_{\tilde{R}}^2} + \Exp{\Vert E\Wshift(k) \Vert_P^2} \\ &+ \Exp{\Vert \Xshift^s(k) \Vert_P^2} - \Exp{\Vert \Xshift^s(k+1) \Vert_P^2} + \cshift \\
                =& \Exp{\Vert E\Wshift(k) \Vert_P^2} + \cshift
            \end{split}
        \end{equation*}
        since $\Ushift^s(k) = K\Xshift^s(k)$ and $\Exp{\Vert \Xshift^s(k) \Vert_P^2} = \Exp{\Vert \Xshift^s(k+1) \Vert_P^2}$ because of the stationarity of $\Xshift^s\args$. This shows equation \eqref{eq:FirstModifiCost} and, thus, proves the lemma.
  \end{proof}

  In Lemma \ref{lem:StagecostShiftedArguments}, we have shown that we can modify the stage costs in a suitable way to obtain a reformulation in terms of shifted arguments. The next lemma shows how we can modify the costs further to derive a dissipativity inequality for the stage costs in the form of \eqref{eq:ShiftedStagecosts}.

  \begin{lemma} \label{lem:DissiStochDet}
    Assume that $(A,Q_1^{1/2})$ is detectable.
    Further, we define the dynamics for the error $\Xdist\args := \Xshift\args - \Xshift^s\args$ between the state process $\Xshift\args$ and the stationary process $\Xshift^s\args$ from Lemma \ref{lem:StagecostShiftedArguments} as
    \begin{align*}
      \Xdist(k+1) &= A\Xdist(k) + B\Udist(k) = A (\Xshift(k)-\Xshift^s(k)) + B(\Ushift(k)-\Ushift^s(k))\\
      \Xdist(0) &= \Xdist_0 = X_0 - \xeq - \Xshift^s_0
    \end{align*}
    with $\Udist\args := \Ushift\args - \Ushift^s\args$.
    Then there exists a solution $(\xeqopt,\ueqopt)$ to $(I-A)x - Bu = 0$, a function $\alpha \in \K_{\infty}$, a symmetric, positive definite matrix $S \in \R^{n \times n}$, and a vector $q \in \R^n$, such that for every $k \in \N_0$ and $(X(k),U(k)) \in \Lp[2]{\R^n} \times \Lp[2]{\R^l}$ the inequality
    \begin{equation*}
      \hat{\ell}(\Xdist(k),\Udist(k)) - \hat{\ell}(\xeqopt,\ueqopt) + \tilde{\lambda}(\Xdist(k)) - \tilde{\lambda}(\Xdist(k+1)) \geq \alpha \left( \Exp{ \Vert \Xdist(k) - \xeqopt \Vert^2 } \right)
    \end{equation*}
    holds with $\tilde{\lambda}(\Xdist(k)) := \Exp{\Vert \Xdist(k) \Vert_S^2 + q^T \Xdist(k)}$.
  \end{lemma}
  \begin{proof}
    Since $(A,Q_1^{1/2})$ is detectable, we know by \cite[Lemma 5.4]{Gruene2018} that there is $\tilde{S} \in \R^{n \times n}$ symmetric and positive definite satisfying the matrix inequality 
    \[Q_1 + \tilde{S} - A^T\tilde{S}A > 0.\]
    For a given $\gamma \in (0,1]$, set $\widetilde{S}_{\gamma} := \gamma \widetilde{S}$ and $ Q_{\gamma} := Q_1 + \widetilde{S}_{\gamma} - A^T \widetilde{S}_{\gamma}A$.
    Then following the calculation of \cite[Lemma 4.1]{Gruene2018} we get that
    \begin{equation*}
        \begin{split}
            \hat{\ell}_{\gamma}(\Xdist(k),\Udist(k)) :=& \hat{\ell}(\Xdist(k),\Udist(k)) + \Exp{\Vert \Xdist(k) \Vert_{\widetilde{S}_{\gamma}}^2} - \Exp{\Vert \Xdist(k+1) \Vert_{\widetilde{S}_{\gamma}}^2} \\
            =& \Exp{\Vert (\Xdist(k), \Udist(k)) \Vert_H^2} + \Exp{ \sshift^T \Xdist(k) + \vshift^T \Udist(k) + \cshift} \\
            &+ \trace \big( Q_2 \covar(\Xdist(k)) + R_2 \covar(\Udist(k)) \big)
        \end{split}
    \end{equation*}
    with
    \begin{equation*}
        H := \dfrac{1}{2}\left( 
        \begin{array}{cc}
            2Q_{\gamma} & \gamma G \\
            \gamma G & 2 R_{\gamma}
        \end{array}
        \right),
    \end{equation*}
    $R_{\gamma} := R_1-B^T\widetilde{S}_{\gamma}B$ and $G := -A^T\widetilde{S}B - B^T\widetilde{S}A$.
    Using the Schur complement, we can show that $H$ is positive definite for a sufficient small $\tilde{\gamma} \in (0,1]$; see \cite[Proof of Lemma 4.1]{Gruene2018}. Thus, we conclude that the function
    \begin{equation}
        h_{\tilde{\gamma}}(x,u) := (x, u)^T H (x,u) +  \sshift^T x + \vshift^T u + \cshift
    \end{equation}
    is strictly convex in $(x,u)$. Then the optimization problem 
    \begin{equation}
        \min_{(x,u) \in \R^n \times \R^l} h_{\tilde{\gamma}}(x,u), \quad \text{s.t.}~x-Ax-Bu=0
    \end{equation}
    admits a unique global solution $(\xeqopt,\ueqopt)$. Applying \cite[Proposition~4.3]{Damm2014} we can deduce the existence of a vector $q \in \R^n$ and a constant $r>0$ such that
    \begin{equation} \label{eq:DissiDet}
        h_{\tilde{\gamma}}(x,u) - h_{\tilde{\gamma}}(\xeqopt,\ueqopt) + q^T x - q^T(Ax+Bu) \geq  r \Vert x-\xeqopt \Vert^2
    \end{equation}
    holds for all $(x,u) \in \R^n \times \R^l$. 
    Further, we observe
    \begin{equation} \label{eq:costEqu}
        \begin{split}
            \hat{\ell}_{\tilde{\gamma}}(\xeqopt,\ueqopt) &= \Exp{h_{\Tilde{\gamma}}(\xeqopt,\ueqopt)} \\
            &= \hat{\ell}(\xeqopt,\ueqopt) + \Exp{\Vert \xeqopt \Vert_{\widetilde{S}_{\tilde{\gamma}}}^2} - \Exp{\Vert A\xeqopt+ B\ueqopt \Vert_{\widetilde{S}_{\tilde{\gamma}}}^2} \\
            &= \hat{\ell}(\xeqopt,\ueqopt) + \Exp{\Vert \xeqopt \Vert_{\widetilde{S}_{\tilde{\gamma}}}^2} - \Exp{\Vert \xeqopt \Vert_{\widetilde{S}_{\tilde{\gamma}}}^2} = \hat{\ell}(\xeqopt,\ueqopt).
        \end{split}
    \end{equation}
    Hence, using the equations \eqref{eq:DissiDet} and \eqref{eq:costEqu} and setting $S := \tilde{S}_{\tilde{\gamma}}$ it follows that 
    \begin{equation*}
        \begin{split}
            &\hat{\ell}(\Xdist(k),\Udist(k)) - \hat{\ell}(\xeqopt,\ueqopt) + \tilde{\lambda}(\Xdist(k)) - \tilde{\lambda}(\Xdist(k+1)) \\
            =& \hat{\ell}_{\tilde{\gamma}}(\Xdist(k),\Udist(k)) - \hat{\ell}(\xeqopt,\ueqopt)  + \Exp{q^T \Xdist(k)} - \Exp{q^T\Xdist(k+1)} \\
            =& \Exp{ h_{\tilde{\gamma}}(\Xdist(k),\Udist(k)) - h_{\tilde{\gamma}}(\xeqopt,\ueqopt) + q^T \Xdist(k) - q^T(A\Xdist(k) + B\Udist(k))} \\
            &\quad + \trace \big( Q_2 \covar(\Xdist(k)) + R_2 \covar(\Udist(k)) \big) \\
            &\geq \Exp{r \Vert \Xdist(k)-\xeqopt \Vert^2} + \trace \big( Q_2 \covar(\Xdist(k)) + R_2 \covar(\Udist(k)) \big) \\
            &\geq r \Exp{\Vert \Xdist(k)-\xeqopt \Vert^2},
        \end{split}
    \end{equation*}
    which proves the claim with $\alpha(s) := rs$.
  \end{proof}

  The following theorem summarizes the results of this section. It shows how the storage function for the stochastic linear-quadratic OCP exactly reads and of which form the stationary pair is.

  \begin{theorem} \label{thm:StochDissi}
    Assume that the equation $(I-A)\xeq - B\ueq = E\mu_W + z$ has a solution $(\xeq,\ueq)$ and $(A,Q_1^{1/2})$ is detectable. 
    Then there exist an invariant distribution $\DistrStatOptX \in \mathcal{P}_2(\R^n)$ and an initial condition $\XstatOpt(0) = {\XstatOpt}_0 \sim \DistrStatOptX$ such that the state and control processes 
    \begin{align*}
      \XstatOpt(k+1) &= f(\XstatOpt(k),\UstatOpt(k),W(k)) \\
      \UstatOpt(k) &= K(\XstatOpt(k) - \xeqopt - \xeq) + \ueqopt + \ueq
    \end{align*}
    define a stationary pair according to Definition \ref{defn:StationaryPair} and the stochastic OCP \eqref{eq:sOCP} is strictly dissipative in $L^2$ at $(\XstatOpt\args,\UstatOpt\args)$. Moreover, the dissipation inequality \eqref{eq:sDI} is satisfied with $\alpha \in \K_{\infty}$ from Lemma~\ref{lem:DissiStochDet} and the storage function
    \begin{equation} \label{eq:FinalStorageFunction}
      \lambda(k,X) := \Exp{\Vert X - (\XstatOpt(k) - \xeqopt) \Vert_{P+S}^2 - \Vert X - \xeq \Vert_P^2 + q^T (X - (\XstatOpt(k) - \xeqopt))}
    \end{equation}
    where $P$ is the solution of the algebraic Riccati equation~\eqref{eq:algRiccati} and $S$, $q$, and $(\xeqopt,\ueqopt)$ are from Lemma~\ref{lem:DissiStochDet}.
  \end{theorem}
  \begin{proof}
    By Lemma \ref{lem:StagecostShiftedArguments}, we know that the state and control processes
    \begin{align*}
      \Xshift^s(k+1) &= A \Xshift^s(k) + B \Ushift^s(k) + E\Wshift(k) \\
      \Ushift^s(k) &= K \Xshift^s(k)
    \end{align*}
    define a stationary pair for the system $f(\Xshift(k),\Ushift(k),\Wshift(k)) -z$ with $\Xshift^s_0 \sim \varrho^s_X \in \mathcal{P}_2(\R^n)$. Now let us assume that $\XstatOpt(k) = \Xshift^s(k) + \xeq + \xeqopt$ for some $k \in \N_0$. Then we get 
    \begin{equation} \label{eq:shiftStatinaryProcess}
      \begin{split}
        \XstatOpt&(k+1) = f(\XstatOpt(k),\UstatOpt(k),W(k)) \\
        =& \; A \XstatOpt(k) + B \UstatOpt(k) + E W(k) + z \\
        =& \; A(\Xshift^s(k) + \xeq + \xeqopt) + B(K(\XstatOpt(k) - \xeqopt - \xeq) + \ueqopt + \ueq) + EW(k) + z \\
        =& \; (A+BK) \Xshift^s(k) + (A\xeqopt + B\ueqopt) + EW(k) + A\xeq + B\ueq + z \\
        =& \; (A+BK) \Xshift^s(k) + \xeqopt + EW(k) + (\xeq - E\mu_W) \\
        =& \; (A+BK) \Xshift^s(k) + \Wshift(k) + \xeq + \xeqopt \\
        =& \; \Xshift^s(k+1) + \xeq + \xeqopt .
      \end{split}
    \end{equation}
    Thus, if we set $\XstatOpt(0) = \Xshift^s_0 + \xeq + \xeqopt$ as the initial condition, we get $\XstatOpt(k) = \Xshift^s(k) + \xeq + \xeqopt$ for all $k \in \N_0$ by induction. Moreover, since $\Xshift^s(k) \sim \varrho_X^s \in \mathcal{P}_2(\R^n)$ for all $k \in \N_0$ as a stationary process we can directly conclude that $\XstatOpt(k) \sim \DistrStatOptX$ for all $k \in \N_0$ where the distribution $\DistrStatOptX\args$ (in the sense of a measure) is given by $\DistrStatOptX\args = \varrho_X^s(\cdot-(\xeq+\xeqopt)) \in \mathcal{P}_2(\R^n)$.
    Moreover, the stationarity of $\UstatOpt\args$ then follows by its feedback form, cf. Remark~\ref{rem:FeedbackStationarity}.
    Hence, $(\XstatOpt\args,\UstatOpt\args)$ defines a stationary pair according to Definition \ref{defn:StationaryPair}. \\
    To infer the second statement of the proof, we use Lemma \ref{lem:StagecostShiftedArguments} and \ref{lem:DissiStochDet} to conclude that 
    \begin{equation} 
      \begin{split}  \label{eq:ProofDissi1}
        \hat{\ell}&(\Xshift(k),\Ushift(k)) - \hat{\ell}(\Xshift^s(k),\Ushift^s(k)) - \hat{\ell}(\xeqopt,\ueqopt) - \cshift \\
        &+ \hat{\lambda}(k,\Xshift(k)) - \hat{\lambda}(k+1,\Xshift(k+1)) + \tilde{\lambda}(\Xdist(k)) - \tilde{\lambda}(\Xdist(k+1))\\
        \geq& \; \alpha \left( \Exp{ \Vert \Xdist(k) - \xeqopt \Vert^2 } \right)
      \end{split}
    \end{equation}
    holds for all $k \in \N_0$ and $(\Xshift(k), \Ushift(k)) \in \Lp[2]{\R^n} \times \Lp[2]{\R^l}$. 
    Furthermore, by Lemma \ref{lem:PushShiftToCost}, the transformations $\Xshift(k) = X(k) - \xeq$, and equation \eqref{eq:shiftStatinaryProcess} we get
    \begin{equation}
        \begin{split}  \label{eq:ProofDissi2}
            &\ell(\XstatOpt(k), \UstatOpt(k)) = \ell(\Xshift^s(k) + \xeq + \xeqopt, \Ushift^s(k) + \ueq + \ueqopt) \\
            &= \hat{\ell}(\Xshift^s(k) + \xeqopt, \Ushift^s(k) + \ueqopt) \\
            &= \Exp{(\Xshift^s(k) + \xeqopt)^T Q_1 (\Xshift^s(k) + \xeqopt)} + \Exp{(\Ushift^s(k) + \ueqopt)^T R_1 (\Ushift^s(k) + \ueqopt)} \\
            &\quad + \Exp{\sshift^T (\Xshift^s(k) + \xeqopt)+ \vshift^T (\Ushift^s(k) + \ueqopt) + \cshift} \\
            &\quad + \trace \big( Q_2 \covar(\Xshift^s(k) + \xeqopt) + R_2 \covar(\Ushift^s(k) + \ueqopt) \big) \\
            &= \hat{\ell}(\Xshift^s(k),\Ushift^s(k)) + \hat{\ell}(\xeqopt,\ueqopt) + \cshift + 2\Exp{\Xshift^s(k)^T Q_1 \xeqopt} + 2\Exp{\Ushift^s(k)^T R_1 \ueqopt} \\
            &= \hat{\ell}(\Xshift^s(k),\Ushift^s(k)) + \hat{\ell}(\xeqopt,\ueqopt) + \cshift
        \end{split}
    \end{equation}
    since $\Xshift^s(k)$ and $\Ushift^s(k)$ have zero mean and $\xeqopt$, $\ueqopt$ as deterministic values are stochastically independent from $\Xshift^s(k)$ and $\Ushift^s(k)$.
    Hence, by using the transformation of the stage costs from Lemma~\ref{lem:PushShiftToCost} again together with $\lambda(k,X(k)) = \hat{\lambda}(k,X(k)) + \tilde{\lambda}(k,X(k))$ and the equations \eqref{eq:ProofDissi1} and \eqref{eq:ProofDissi2} 
    we receive
    \begin{equation*}
      \begin{split}
        &\ell(X(k),U(k)) - \ell(\XstatOpt(k),\UstatOpt(k)) + \lambda(k,X(k)) - \lambda(k+1,X(k+1)) \\
        =& \hat{\ell}(\Xshift(k),\Ushift(k)) - \hat{\ell}(\Xshift^s(k),\Ushift^s(k)) - \hat{\ell}(\xeqopt,\ueqopt) - \cshift + \lambda(k,X(k)) - \lambda(k+1,X(k+1)) \\
        \geq& \alpha \left( \Exp{ \Vert (X(k) - \xeq) - \Xshift^s(k) - \xeqopt \Vert^2 } \right) = \alpha \left( \Exp{ \Vert X(k) - \XstatOpt(k) \Vert^2 } \right)
      \end{split}
    \end{equation*}
    for all $k \in \N_0$ and $(X(k), U(k)) \in \Lp[2]{\R^n} \times \Lp[2]{\R^n}$ with $\lambda(k,X)$ as given in equation \eqref{eq:FinalStorageFunction}. 
    This proves the theorem, since $\lambda(k,X)$ is bounded from below because of the positive definiteness of $S$.
  \end{proof}

  \begin{remark} The assumptions of Theorem \ref{thm:StochDissi} are not restrictive. The non-existence of $(x^s,u^s)$ implies that the expectation of $X(k)$ cannot remain constant in $k$, hence this condition is necessary for the existence of a pair of stationary processes. The detectability condition is known to be necessary for strict dissipativity in the deterministic case \cite{Gruene2018}, which is a particular special case of our stochastic setting, as Example \ref{ex:det} shows.
  \end{remark}

  At the end of this section, we want to give a small example, which shows that our stochastic dissipativity results are consistent with the deterministic theory.

  \begin{example}\label{ex:det}
    Consider a stochastic optimal control problem without additive noise
    \begin{equation*}
      \begin{split}
        \min_{U\args} J_N(X_0,U) := \sum_{k=0}^{N-1} \Exp{X^TQX + U^TRU + r^TX + v^TU + c} \\
        s.t ~ X(k+1) = AX(k) + BU(k), \quad X(0) = X_0.
      \end{split}
    \end{equation*}
    Since $W\args = 0$ and $z=0$, we do not have to apply the shift from Theorem~\ref{lem:PushShiftToCost} and, thus, the stochastic OCP is strictly dissipative in $L^2$ at $(\XstatOpt\args,\UstatOpt\args) = (\xeqopt,\ueqopt)$ with storage function
    \begin{equation}
      \lambda(k,X) = \Exp{\Vert X \Vert_{P+S}^2 - \Vert X \Vert_P^2 + q^T X} = \Exp{\Vert X \Vert_S + q^T X}
    \end{equation}
    according to Theorem \ref{thm:StochDissi}. Further, if we additionally restrict ourselves to Dirac distributions for $X(k)$ and $U(k)$, i.e., $X(k) \in \R^n$ and $U(k) \in \R^l$ has to hold for all $k \in \N_0$, then the storage function is the same as given in \cite{Gruene2018} for the deterministic generalized linear-quadratic optimal control problem.
  \end{example}

\section{Stochastic turnpike properties} \label{sec:Turnpike}
  In this section, we show that the strict dissipativity in $L^2$ from Theorem \ref{thm:StochDissi} implies turnpike properties of our stochastic OCP similar to Lemma \ref{lem:DetTurnpike} for deterministic problems. 
  The straightforward extension of the deterministic turnpike property from Definition \ref{defn:Turnpike} to stochastic systems would be given by replacing the steady state $(\xeq,\ueq)$ with a stationary pair $(\Xstat\args,\Ustat\args)$ in Definition~\ref{defn:Turnpike}.
  However, unlike the deterministic case, in stochastic settings turnpike behavior can be formulated and observed in different objects like distributions, moments, or sample paths of the stochastic system, cf.~\cite{Ou2021}. 
  In order to conclude these different turnpike phenomena, we can choose appropriate distance measures between the stochastic processes. Note that in contrast to the deterministic case on $\R^n$, different norms for random variables are not necessarily equivalent.
  As we will see and in contrast to \cite{Gros2022}, our dissipativity notion is strong enough to not only conclude a turnpike behavior of the stationary distribution and moments but also for the single realization paths. This pathwise turnpike behavior can be characterized by the $L^2$ norm and in an appropriate probabilistic sense.
  Before presenting these results, we define the rotated costs used in several of the following proofs.

  \begin{definition} \label{defn:rotatedCost}
    For a stochastic OCP of form \eqref{eq:sOCP}, which is strictly dissipative in $L^2$ at $(\XstatOpt\args,\UstatOpt\args)$ with storage function $\lambda$, we define the rotated stage costs as 
    \begin{equation}
      \tilde{\ell}(k,X,U) := \ell(X,U) - \ell(\XstatOpt(k),\UstatOpt(k)) + \lambda(k,X) - \lambda(k+1,f(X,U,W(k)))
    \end{equation}
    and the rotated cost as 
    \begin{equation}
      \tilde{J}_N(X_0,U\args) := \sum_{k=0}^{N-1} \tilde{\ell}(k,X(k),U(k)).
    \end{equation}
  \end{definition}

  The next theorem starts our turnpike analysis by showing that the strict dissipativity in $L^2$ directly implies turnpike behavior in $L^2$ following the same arguments as in the deterministic case. 
  
  \begin{theorem} \label{thm:TurnpikeL2}
    Let the assumptions of Theorem \ref{thm:StochDissi} hold. Then for every $X_0$, there exists a constant $C \in \R$ such that for each $\delta > 0$, each control process $U\args$ satisfying $J_N(X_0,U\args) \leq \delta + \sum_{k=0}^{N-1} \ell(\XstatOpt(k),\UstatOpt(k))$ and each $\eps > 0$ the value 
    \begin{equation*}
      \Lset_{\eps} := \#\Big\lbrace k \in \lbrace 0,\ldots,N-1 \rbrace \mid \E \big[ \norm{ X_U(k,X_0) - \XstatOpt(k) }^2 \big] \leq \eps \Big\rbrace
    \end{equation*}
    satisfies the inequality $\Lset_{\eps} \geq N - (\delta + C)/\alpha(\eps)$ where $\alpha$ is the $\K_{\infty}$-function from Lemma~\ref{lem:DissiStochDet}.
  \end{theorem}
  \begin{proof}
    The proof follows the same arguments as \cite[Theorem 5.3]{Gruene2013}.
    Set $C := \lambda(0,X_0) - M$ where $M \in \R$ is a lower bound on $\lambda$ from Theorem \ref{thm:StochDissi}. Then for $\tilde{J}_N(X_0,U) := \sum_{k=0}^{N-1} \tilde{\ell}(X(k),U(k))$ from Definition~\ref{defn:rotatedCost} we get 
    \begin{equation} \label{eq_turnpikeContr}
        \begin{split}
            \tilde{J}_N(X_0,U) =& \; \sum_{k=0}^{N-1} \Ell(X(k),U(k)) - \Ell(\XstatOpt(k),\UstatOpt(k)) \\
            &\quad + \lambda(k,X(k)) - \lambda(k+1,f(X(k),U(k),W(k))) \\ 
            =& \; J_N(X_0,U) - \sum_{k=0}^{N-1} \ell(\XstatOpt(k),\UstatOpt(k)) + \lambda(0,X(0)) - \lambda(N,X(N)) \\
            \leq& \; \delta + C.
        \end{split}
    \end{equation}
    Now assume that $\Lset_{\eps} < N-(\delta+C)/\alpha(\eps)$ with $\alpha$ from Theorem \ref{thm:StochDissi}. This means there is a set $\mathcal{N} \subset \lbrace 0,\ldots,N-1 \rbrace$ of $N-\Lset_{\eps} > (\delta+C)/\alpha(\eps)$ time instants such that 
    $\E [ \Vert X(k) - \XstatOpt(k) \Vert^2 ] \geq \eps$ for all $k \in \mathcal{N}$. Using Theorem \ref{thm:StochDissi}, this implies 
    \begin{equation*}
        \begin{split}
            \tilde{J}_N(X_0,U) \geq& \sum_{k=0}^{N-1} \alpha\left(\E\big[\Vert X(k) - \XstatOpt(k) \Vert^2\big]\right) \geq \sum_{k \in \mathcal{N}} \alpha\left(\E\big[\Vert X(k) - \XstatOpt(k) \Vert^2\big]\right) \\
            >& \dfrac{\delta+C}{\alpha(\eps)} \alpha(\eps) = \delta + C
        \end{split}
    \end{equation*}
    which contradicts \eqref{eq_turnpikeContr} and, thus, proves the theorem.
  \end{proof}

  \begin{remark}
    The condition $J_N(X_0,U(\cdot)) \leq \delta + \sum_{k=0}^{N-1} \ell(X_{\star}^s(k),U_{\star}^s(k))$ demands that the values of the trajectories considered in Theorem~\ref{thm:TurnpikeL2} are close to the stationary values, at least in the long run when $\delta>0$ is large. This condition is analogous to the deterministic setting, e.g., in \cite[Definition~2.2]{Gruene2016}, where the corresponding trajectories are called ``near steady state solutions''.
  \end{remark}
  
  Theorem~\ref{thm:TurnpikeL2} shows that for all $N \in \N$ the solutions of the stochastic OCP \eqref{eq:sOCP} have to be close to the stationary process $\XstatOpt\args$ except for a uniformly bounded number of time-instances. Here the distance between two stochastic processes is measured in the mean-square distance. Even if it is likely that the mean-square distance allows statements about the pathwise behavior of the processes because it measures the deviation between the random variables, it is not immediately obvious what Theorem \ref{thm:TurnpikeL2} means for the single realization paths of $X(\cdot)$, e.g., one cannot infer an almost sure turnpike properties for the path. However, one may conclude a pathwise turnpike in probability which is formalized in the following theorem.

  \begin{theorem} \label{thm:TurnpikeProb}
    Let the assumptions of Theorem \ref{thm:StochDissi} hold. Then for every $X_0$, there exists a constant $C \in \R$ such that for each $\delta > 0$, each control process $U\args$ satisfying $J_N(X_0,U\args) \leq \delta + \sum_{k=0}^{N-1} \ell(\XstatOpt(k),\UstatOpt(k))$ and each $\eps, \eta > 0$ the value 
    \begin{equation*}
      \Pset_{\eps,\eta} := \#\Big\lbrace k \in \lbrace 0,\ldots,N-1 \rbrace \mid \Prob \big( \norm{ X_U(k,X_0) - \XstatOpt(k) } \geq \eps \big) \leq \eta \Big\rbrace
    \end{equation*}
    satisfies the inequality $\Pset_{\eps,\eta} \geq N - (\delta + C)/\alpha(\eps^2\eta)$ where $\alpha$ is the $\K_{\infty}$-function from Lemma~\ref{lem:DissiStochDet}.
  \end{theorem}
  \begin{proof}
    Using the Markov inequality, we get 
    \begin{equation} \label{eq:MarkovIneq}
        \Prob\left(\Vert X(k) - \XstatOpt(k) \Vert \geq \epsilon \right) \leq \dfrac{1}{\eps^2}\Exp{\Vert X(k) - \XstatOpt(k) \Vert^2} .
    \end{equation}
    Further, by Theorem \ref{thm:TurnpikeL2}, we know that there exist a function $\alpha \in \K_{\infty}$ with $\alpha(\eps^2\eta)$ such that there are at least $N-(\delta+C)/\alpha(\eps^2\eta)$ time instants for which $\E [ \Vert X(k) - \XstatOpt(k) \Vert^2 ] \leq \eps^2 \eta$.
    Using equation \eqref{eq:MarkovIneq}, this gives 
    \[\Prob\left(\Vert X(k) - \XstatOpt(k) \Vert \geq \eps \right) \leq \eta\] 
    for all these time instants and, thus, proves the claim.
  \end{proof}
  The pathwise turnpike behavior from Theorem~\ref{thm:TurnpikeProb} says that the probability that a single realization is not near $\XstatOpt(\cdot)$ is small, except possibly at certain time instances whose number is independent of $N$. Thus, this theorem gives us a more descriptive statement about the pathwise behavior of the state processes.
  Another object besides the paths that is often of interest regarding random variables is their distributions. However, there are many ways to measure the distance between two distributions, cf. \cite[p.20]{DasGupta2008}. So, before starting our turnpike analysis concerning the distributions, we must specify the metric to measure the distance between them. The following definition introduces the Wasserstein metric, which is closely related to the $L^p$-norm and, thus, a natural choice for our analysis. 

  \begin{definition}[\hspace{-0.1pt}{\cite[Definition~6.1]{Villani2009}}]
    For $p \in (0,\infty]$ and two random variables $X,Y \in L^p(\Omega,\F,\mathbb{P};\R^n)$ we define the \emph{Wasserstein distance of order $p$} as 
    \begin{equation}
      W_p(X,Y) := \inf \left\lbrace \Vert \bar{X} - \bar{Y} \Vert_{L^p}, ~ \bar{X} \sim X, ~ \bar{Y} \sim Y \right\rbrace .
    \end{equation}
  \end{definition}

  Note that although we formally write $W_p(X,Y)$ for two random variables to simplify the notation, the Wasserstein metric is only a metric on the space of probability measures and not on the space of random variables.
  \begin{remark}
      It can be shown that $W_2$ is a finite metric on the space of probability measures $\mathcal{P}_2(\R^n)$; see \cite{Villani2009}. Thus, if we say in the following that two distributions are the same, i.e., $\varrho_1 = \varrho_2$ for $\varrho_1, \varrho_2 \in \mathcal{P}_2(\R^n)$, then this means that their distance in the Wasserstein metric of order $2$ is zero. 
  \end{remark}
  
  The succeeding theorem shows that the turnpike behavior in $L^2$ also implies a turnpike property of distributions if their distance is measured by the Wasserstein metric.  

  \begin{theorem} \label{thm:TurnpikeDistr}
    Let the assumptions of Theorem \ref{thm:StochDissi} hold. Then for every $X_0$, there exists a constant $C \in \R$ such that for each $\delta > 0$, each control process $U\args$ satisfying $J_N(X_0,U\args) \leq \delta + \sum_{k=0}^{N-1} \ell(\XstatOpt(k),\UstatOpt(k))$ and each $\eps > 0$ the value 
    \begin{equation*}
      \Dset_{\eps} := \#\Big\lbrace k \in \lbrace 0,\ldots,N-1 \rbrace \mid W_2 \left( X_U(k,X_0), \XstatOpt(k) \right) \leq \eps \Big\rbrace
    \end{equation*}
    satisfies the inequality $\Dset_{\eps} \geq N - (\delta + C)/\alpha(\sqrt{\eps})$ where $\alpha$ is the $\K_{\infty}$-function from Lemma~\ref{lem:DissiStochDet}.
  \end{theorem}
  \begin{proof}
    Since $W_2 \left( X_U(k,X_0), \XstatOpt(k) \right) \leq \Vert X_U(k,X_0) - \XstatOpt(k) \Vert_{L^2}$ holds per definition and $\Vert X_U(k,X_0) - \XstatOpt(k) \Vert_{L^2}^2 = \Exp{\Vert X_U(k,X_0) - \XstatOpt(k) \Vert^2} \leq \eps$ implies that $\Vert X_U(k,X_0) - \XstatOpt(k) \Vert_{L^2} \leq \sqrt{\eps}$ holds for all $\eps > 0$, the claim follows directly by Theorem~\ref{thm:TurnpikeL2}.
  \end{proof}

  As mentioned before, the Wasserstein distance is not the only metric in which turnpike behavior in distribution can be characterized since there are several other metrics on measure spaces that are not necessarily equivalent, e.g., the Kullback-Leibler distance or the total variation metric. 
  However, since the Wasserstein metric is a natural lower bound on the $L^p$ norm by its definition, and this is crucial for the proof of Theorem~\ref{thm:TurnpikeDistr}, from our dissipativity notion we can, in general, not conclude a turnpike property in the other mentioned metrics.
  However, the Wasserstein metric is one of the stronger metrics for distributions since it delivers upper bounds for other metrics like the L\'evy-Prokhorov metric, which characterizes the weak convergence of measures. This allows us to make additional statements about the behavior of moments, as shown in the next theorem.

  \begin{theorem} \label{thm:TurnpikeMoments}
    Let the assumptions of Theorem \ref{thm:StochDissi} hold. Then for every $X_0$, there exists a constant $C \in \R$ such that for each $\delta > 0$, each control process $U\args$ satisfying $J_N(X_0,U\args) \leq \delta + \sum_{k=0}^{N-1} \ell(\XstatOpt(k),\UstatOpt(k))$ and each $\eps > 0$ the values
    \begin{gather*}
      \Mset^1_{\eps} := \#\Big\lbrace k \in \lbrace 0,\ldots,N-1 \rbrace \mid \left\Vert\Exp{X_U(k,X_0)} - \Exp{\XstatOpt(k)} \right\Vert \leq \eps \Big\rbrace, \\
      \Mset^2_{\eps} := \#\Big\lbrace k \in \lbrace 0,\ldots,N-1 \rbrace \mid \left\vert \sqrt{\trace(\covar(X_U(k,X_0)))}- \sqrt{\trace(\covar(\XstatOpt(k)))} \right\vert \leq \eps \Big\rbrace
    \end{gather*}
    satisfy the inequalities $\Mset^1_{\eps} \geq N - (\delta + C)/\alpha(\sqrt{\eps})$ and $\Mset^2_{\eps} \geq N - (\delta + C)/\alpha(\sqrt{2\eps})$, where $\alpha$ is the $\K_{\infty}$-function from Lemma~\ref{lem:DissiStochDet}.
  \end{theorem}
  \begin{proof}
   Using $ W_2(X,Y) \geq W_1(X,Y)$, cf. \cite[Remark 6.6]{Villani2009}, and $\Exp{\Vert X - Y \Vert} \geq \Vert \Exp{ X - Y } \Vert$ for all $X,Y \in \Ltwo{\R^n}$ we get 
    \begin{equation} \label{eq:BoundNormExpectation}
        \begin{split}
            W_2(X,Y) &\geq W_1(X,Y) = \inf \left\lbrace \Exp{\Vert \bar{X} - \bar{Y} \Vert}, ~ \bar{X} \sim X, ~ \bar{Y} \sim Y \right\rbrace \\
            &\geq \inf \left\lbrace \Vert \Exp{ \bar{X} - \bar{Y} } \Vert, ~ \bar{X} \sim X, ~ \bar{Y} \sim Y \right\rbrace = \Vert \Exp{ X } - \Exp { Y } \Vert.
        \end{split}
    \end{equation}
    Thus, the first part of the theorem follows directly by Theorem \ref{thm:TurnpikeDistr}. Further, by the triangle inequality we get 
    \begin{equation*}
      W_2(X,\E[X]) \leq W_2(X,Y) + W_2(Y,\E[Y]) + W_2(\E[X],\E[Y])
    \end{equation*}
    and 
    \begin{equation*}
      W_2(Y,\E[Y]) \leq W_2(X,Y) + W_2(X,\E[X]) + W_2(\E[X],\E[Y]),
    \end{equation*}
    which implies 
    \begin{equation} \label{eq:BoundVariance}
      \vert W_2(X,\E[X]) - W_2(Y,\E[Y]) \vert \leq W_2(X,Y) + W_2(\E[X],\E[Y]).
    \end{equation}
    Additional, it holds that
    \begin{equation*}
         W_2(X,\E[X]) = \sqrt{\trace(\covar(X))} \mbox{ and } W_2(\E[X],\E[Y]) = \Vert \E[X] - \E[Y] \Vert
    \end{equation*}
    for all $X \in \Ltwo{\R^n}$.
    Hence, by inequalities~\eqref{eq:BoundNormExpectation} and \eqref{eq:BoundVariance} we get 
    \begin{equation*}
        \left\vert \sqrt{\trace(\covar(X_U(k,X_0)))}- \sqrt{\trace(\covar(\XstatOpt(k)))} \right\vert \leq 2 W_2(X,Y),
    \end{equation*}
    which together with Theorem~\ref{thm:TurnpikeDistr} shows the second part of the theorem.
  \end{proof}

  \begin{remark}
      The turnpike property for the expectation value from Theorem~\ref{thm:TurnpikeMoments} can be seen as a discrete-time and additive-noise version of the turnpike property presented in \cite[Theorem~5.1]{Sun2022} for the state process, where $x^*$ in \cite{Sun2022} corresponds to $\Exp{\XstatOpt\args}$ in our result. In the setting of \cite{Sun2022} as well as in the deterministic linear-quadratic setting treated in  \cite{Gruene2018}, the turnpike property is actually exponential, which in \cite{Gruene2018} was shown using the techniques from \cite{Damm2014}. 
      Although we conjecture that the approach presented in \cite{Damm2014} could be used in our setting to obtain an exponential version of the stochastic turnpike properties, too, we leave the technical details of the proof as an open research question since the extension of the verification of the assumptions from \cite[Theorem~5.6]{Damm2014} in \cite[Theorem~3.3]{Gruene2018} to our stochastic setting is nontrivial and exceeds the scope of this paper.
  \end{remark}

\section{Optimal stationarity} \label{sec:StatOpt}
  The construction of the stationary process in Section \ref{sec:Dissi} is mainly based on the structure of the linear-quadratic stochastic OCP. For the extension of our results to other classes of optimal control problems it would be desirable to have a characterization of the stationary process that does not rely on the linear-quadratic problem structure. In this section we show that this is possible if the stationary control is generated by a state feedback law, cf.\ Definition \ref{def:statfeed}. For this, recall the set of functions $\mathbb{F}(\R^n,\R^l)$ defined in Remark \ref{rem:FeedbackStationarity}.

  \begin{definition}\label{def:statfeed}
    A state process $\Xstat(\cdot)$ is called a stationary process in feedback form for system~\eqref{eq:StochSystem} if there is a stationary distribution $\varrho_X^s \in \mathcal{P}_2(\R^n)$ and a feedback law $\pi^s \in F(\R^n,\R^l)$ such that
    \begin{equation*}
      \Xstat(k) \sim \varrho_X^s, ~ \mbox{and} ~ \Xstat(k+1) = f(\Xstat(k),\Ustat(k),W(k)) \mbox{ for } \Ustat(k) = \pi^s \circ \Xstat(k)
    \end{equation*}
    for all $k \in \N_0$.
  \end{definition}

  Here we use the notation $U = \pi \circ X$ to highlight that $U$ is again a random variable.
  We note that for each stationary process in feedback form according to Definition \ref{def:statfeed} the pair $(\Xstat(\cdot),\Ustat(\cdot))$ with $\Ustat(\cdot)=\pi^s\circ \Xstat(\cdot)$ is a pair of stationary processes in the sense of Definition \ref{defn:StationaryPair} and that the stationary process obtained in Theorem \ref{thm:StochDissi} is in feedback form with $\pi^s(x)=Kx$, which is easily seen to lie in $\mathbb{F}(\R^n,\R^l)$.
  
  The following theorem now shows that the stationary distribution is the unique solution of a stationary optimization problem. Its proof does only use the dissipativity and not the linear-quadratic structure of the problem under consideration.

  \begin{theorem} \label{thm:StationaryOptimality}
    Let the assumptions of Theorem~\ref{thm:StochDissi} hold, i.e. the stochastic OCP \eqref{eq:sOCP} is strictly dissipative in $L^2$. Then the distribution $\DistrStatOptX$ and the feedback $\pi^*(X) = KX$ are a solution of the optimization problem
    \begin{equation} \label{eq:StationaryOptimality}
      \begin{split}
        &\min_{\pi,\varrho_X} \ell(X,U) \\
        s.t. ~ &\exists ~ X \in \Ltwo{\R^n} \sim \varrho_X, ~ W \in \Ltwo{\R^m} \sim \varrho_W \mbox{ independent,}\\
        &\mbox{such that } f(X,U,W) \sim \varrho_X, \mbox{ with } U = \pi \circ X, ~\pi \in \mathbb{F}(\R^n,\R^l).
      \end{split}
    \end{equation}
    Further the stationary distribution $\DistrStatOptX \in  \mathcal{P}_2(\R^n)$ is the unique (partial) solution of this problem, i.e., for every other solution $(\bar{\varrho}_X^s, \bar{\pi}) \in \mathcal{P}_2(\R^n) \times \mathbb{F}(\R^n,\R^l)$ of \eqref{eq:StationaryOptimality} we get $\bar{\varrho}_X^s = \DistrStatOptX$.
  \end{theorem}
  \begin{proof}
    We prove the claim by contradiction. 
    We first observe that the stage cost $\ell(X,U)$ from \eqref{eq:stochStagecost} only depends on the distribution $\varrho_X$ of $X$ since the distribution of $U=\pi \circ X$ with $\pi \in \mathbb{F}(\R^n,\R^l)$ is determined by $\varrho_X$. 
    Thus, problem \eqref{eq:StationaryOptimality} is well posed and for the stationary pair $(\XstatOpt\args,\UstatOpt\args)$ from Theorem~\ref{thm:StochDissi} with $\XstatOpt(k) \sim \DistrStatOptX$ and $\UstatOpt(k) = \pi^* \circ \XstatOpt(k)$ we can define $C_{\ell}^{\star} := \ell(\XstatOpt(k),\UstatOpt(k))$, which is constant for all $k \in \N_0$ because of the stationarity of $\XstatOpt\args$.
    Now assume that there are $\bar{\varrho}_X^s$ and $\bar{\pi}$ solving \eqref{eq:StationaryOptimality} such that $\ell(\bar{X},\bar{U}) \leq C_{\ell}^{\star} $ for $\bar{X} \sim \bar{\varrho}_X^s$ and $\bar{U} = \bar{\pi}(\bar{X})$. 
    We know that the system 
    \begin{equation} \label{eq:SystemStationaryBar}
            \bar{X}^s(k) = f(\bar{X}^s(k),\bar{U}^s(k),W(k)), \quad \bar{U}^s(k) = \bar{\pi}\circ \bar{X}^s(k) 
    \end{equation}
    defines a time-homogeneous Markov chain for all \emph{i.i.d} sequences $\{W(k) \}_{k \in \N}$ with $W(k)$ stochastically independent of $\bar{X}^s(k)$. Thus, for all these sequences $\{W(k) \}_{k \in \N}$ there exists a transition operator $\mathcal{T}_{\bar{\pi}}$ only depending on the distribution $\rho_W$ of the noise such that 
    \begin{equation} \label{eq:transitionOperator}
        \varrho_{\bar{X}^s}(k+1) = \mathcal{T}_{\bar{\pi}} \varrho_{\bar{X}^s}(k)
    \end{equation}
    where $\bar{X}^s(k) \sim \varrho_{\bar{X}^s}(k)$ and $\bar{X}^s(k+1) \sim \varrho_{\bar{X}^s}(k+1)$. 
    Moreover, by the stationarity condition of \eqref{eq:StationaryOptimality} it directly follows that $\bar{\varrho}_X^s$ is a steady state of the transition operator $\mathcal{T}_{\bar{\pi}}$, i.e., $\bar{\varrho}_X^s = \mathcal{T}_{\bar{\pi}} \bar{\varrho}_X^s$, and, thus, for every $X \sim \bar{\varrho}_X^s$ and $W \sim \varrho_W$ with $X$ and $W$ independent it follows $f(X,\bar{\pi}(X),W) \sim \bar{\varrho}_X^s$.
    Then because the constraints of \eqref{eq:StationaryOptimality} ensure that at least one $\bar{X}^s_0 \sim \bar{\varrho}_X^s$ exists, system \eqref{eq:SystemStationaryBar} defines a stationary pair according to Definition~\ref{defn:StationaryPair} with $\ell(\bar{X}^s(k),\bar{U}^s(k)) = \ell(\bar{X},\bar{U}) =: \bar{C}_{\ell}$.
    Hence, using the rotated costs from Definition~\ref{defn:rotatedCost} we can conclude that 
    \begin{equation} \label{eq:rotatedCostContradict}
      \begin{split}
        \tilde{J}_N(\bar{X}^s(k),\bar{U}^s\args) =& \sum_{k=0}^{N-1} (\bar{C}_{\ell} - C_{\ell}^{\star}) + \lambda(0,\bar{X}^s(0)) - \lambda(N,\bar{X}^s(N)) \\
        \leq& N (\bar{C}_{\ell} - C_{\ell}^{\star}) + C_{\lambda} \leq C_{\lambda} 
      \end{split}
    \end{equation}
    holds for all $N \in \N$ with $\lambda$ from Theorem \ref{thm:StochDissi} and $C_{\lambda} := \lambda(0,\bar{X}^s(0)) - M$ where $M \in \R$ is a lower bound on $\lambda$. 
    Let us now first assume that $\bar{C}_{\ell} < C_{\ell}^{\star}$. Then from equation \eqref{eq:rotatedCostContradict} we get that there is an $N_0 \in \N$ such that $\tilde{J}_{N_0}(\bar{X}^s(k),\bar{U}^s\args) < 0$, which is a contradiction, since $\tilde{J}_{N_0}(\bar{X}^s(k),\bar{U}^s\args) \geq 0$ must hold because of the dissipativity of problem \eqref{eq:sOCP}. Thus, $(\DistrStatOptX, \pi^*)$ is a solution of the optimization problem \eqref{eq:StationaryOptimality} and we can conclude that $\bar{C}_{\ell} = C_{\ell}^{\star}$ for all solutions of \eqref{eq:StationaryOptimality}.\\
    To prove the uniqueness of the stationary distribution $\DistrStatOptX$ let us assume that $(\bar{\varrho}_X^s, \bar{\pi})$ is a solution of \eqref{eq:StationaryOptimality} with $\bar{\varrho}_X^s \neq \DistrStatOptX$.
    Then, we know by the strict dissipativity of the stochastic OCP from Theorem~\ref{thm:StochDissi} that $\tilde{\ell}(\bar{X}^s(k),\bar{U}^s(k)) \geq \alpha(\Vert \bar{X}^s(k) - \XstatOpt(k) \Vert_{L^2}^2)$ holds for some $\alpha \in \K_{\infty}$ and, thus, $ \tilde{J}_N(X_0,\bar{U}^s\args) \geq N \alpha(\Vert \bar{X}^s(k) - \XstatOpt(k) \Vert_{L^2}^2)$ for all $N \in \N$. 
    Further, we know, per the definition of the Wasserstein metric that 
    $$\Vert \bar{X}^s(k) - \XstatOpt(k) \Vert_{L^2} \geq W_2( \bar{X}^s(k) - \XstatOpt(k)) := C_W > 0$$
    where $C_W$ is constant for all $k \in \N_0$ because of the stationarity conditions. Therefore, inequality \eqref{eq:rotatedCostContradict} implies $N \alpha((C_W)^2) \leq C_{\lambda}$ for all $N \in \N$, which yields a contradiction since $\alpha((C_W)^2) > 0$ and thus proves the claim.
  \end{proof}

  \begin{remark} \label{rem:StationaryExpectation}
    In \cite[Problem~O]{Sun2022} the authors show that the optimal stationary expectation can be obtained by solving a stationary optimization problem over $\R^n \times \R^l$. This approach is related to our result from Theorem~\ref{thm:StationaryOptimality}, since it gives us a way to compute the expectation of the stationary distribution characterized by the minimizer of problem \eqref{eq:StationaryOptimality}.
    To observe this, we first note that by the identity $X^s_{\star}(\cdot) = \hat{X}^s_{\star}(\cdot) + \tilde{x}^s_{\star}+x^s$ from equation \eqref{eq:shiftStatinaryProcess} with $\E[\hat{X}^s_{\star}(\cdot)] = 0$, we get that $\E[X^s_{\star}(k)] = \tilde{x}^s_{\star}+x^s$ and $\E[U^s_{\star}(k)] = \tilde{u}^s_{\star}+u^s$ holds for all $k \in \N$. Furthermore, 
    by the arguments of the proof of Lemma~\ref{lem:DissiStochDet} we know that $(\tilde{x}^s_{\star},\tilde{u}^s_{\star})$ can be computed by solving 
    \begin{equation}
        \min_{(x,u) \in \R^n \times \R^l} \hat{\ell}(x,u), \quad \text{s.t.}~x-Ax-Bu=0
    \end{equation}
    since $h_{\tilde{\gamma}}(x,u)=\hat{\ell}(x,u)$ for all $(x,u) \in \R^n \times \R^l$ with $x-Ax-Bu=0$, cf.\ equation \eqref{eq:costEqu}.
    Thus, the pair $(\E[X^s_{\star}(k)],\E[U^s_{\star}(k)]) = (\tilde{x}^s_{\star}+x^s,\tilde{u}^s_{\star}+u^s)$ can be obtained by solving 
    \begin{equation}
        \min_{(x,u) \in \R^n \times \R^l} \ell(x,u), \quad \text{s.t.}~x-Ax-Bu-E\mu_w-z=0,
    \end{equation}
    which is the discrete-time version of \cite[Problem~O]{Sun2022} in our setting. However, we want to emphasize that unlike Theorem~\ref{thm:StationaryOptimality} this characterization of the expected values strongly relies on the linear-quadratic structure of the problem and not merely on its dissipativity.
  \end{remark}

  \begin{remark} \label{rem:StationarityGrosZanon}
      Note that we can rewrite the optimization problem \eqref{eq:StationaryOptimality} in terms of the transition operator $\mathcal{T}_{\pi}$ from \eqref{eq:transitionOperator} as 
      \begin{equation}
          \begin{split}
              &\min_{\pi, \varrho_X} \ell(X,U) \\
              s.t.~& X \sim \varrho_X, ~ U = \pi \circ X, \quad  \varrho_X = \mathcal{T}_{\pi} \varrho_X \\
              &\varrho_X \in \mathcal{M}, ~ \pi \in \mathbb{F}(\R^n,\R^l), \\
          \end{split}
      \end{equation}
      where $\mathcal{M}$ is the set of all probability measures $\varrho_X \in \mathcal{P}(\R^n)$ such that there is a random variable $X \in \Ltwo{\R^n}$ with $X \sim \varrho$. Hence, the distributions and the feedback law defining our stationary pair can be characterized by the same optimization problem as the steady state in \cite{Gros2022}. This result can also be seen as a generalization of the uniqueness result of the optimal equilibrium for strictly dissipative deterministic optimal control problems.
  \end{remark}

  Theorem \ref{thm:StationaryOptimality} has shown that the distribution of the stationary state process is unique and can be computed by solving a stationary optimization problem. In general, however, we cannot conclude the uniqueness in $L^2$. Still, we can conclude that the $L^2$-distance between two stationary state processes has to become arbitrarily small for $k \to \infty$. This means that the pathwise behavior of different stationary pairs is nearly identical in the long run, which is sufficient for our turnpike analysis in Section~\ref{sec:Turnpike} since turnpike properties are associated with the long-time behavior of optimal solutions. The following lemma formalizes this convergence.

  \begin{lemma}
    Let the assumptions of Theorem~\ref{thm:StochDissi} hold and let  $(\bar{\varrho}_X^s,\bar{\pi})$ be a solution of the optimization problem \eqref{eq:StationaryOptimality}. Then for all stationary pairs $(\XstatOptbar\args,\UstatOptbar\args)$ in feedback form according to Definition \ref{def:statfeed} with $\UstatOptbar(k) = \bar{\pi} \circ \XstatOptbar(k)$ and $\XstatOptbar(k) \sim \bar{\varrho}^s_X$ we get the convergence 
    \begin{equation}
      \Vert \XstatOptbar(k) - \XstatOpt(k) \Vert_{L^2} \to0 \mbox{ for } k \to \infty.
    \end{equation}
  \end{lemma}
  \begin{proof}
    Using the same notation and arguments as in the proof of Theorem~\ref{thm:StationaryOptimality}, we obtain
    \begin{equation} \label{eq:upperBoundRC}
      \tilde{J}_N(\bar{X}^s_{\star}(0),\bar{U}^s_{\star}\args) \leq N (\bar{C}_{\ell} - C_{\ell}^{\star}) + C_{\lambda} = C_{\lambda} 
    \end{equation}
    for all $N\in \N$ where $\bar{C}_{\ell} = C_{\ell}^{\star}$ because $(\bar{\varrho}_X^s,\bar{\pi})$ is a solution of the optimization problem \eqref{eq:StationaryOptimality}. Further, we know by the strict dissipativity of the stochastic OCP that for all $N \in \N$, the estimate 
    \begin{equation} \label{eq:lowerBoundRC}
      \tilde{J}_N(\bar{X}^s_{\star}(0),\bar{U}^s_{\star}\args) \geq \sum_{k=0}^{N-1} \alpha(\Vert \XstatOptbar(k) - \XstatOpt(k) \Vert_{L^2}^2)
    \end{equation}
    holds. Thus, by combining equations \eqref{eq:upperBoundRC} and \eqref{eq:lowerBoundRC} we get 
    \begin{equation*}
      \sum_{k=0}^{N-1} \alpha(\Vert \XstatOptbar(k) - \XstatOpt(k) \Vert_{L^2}^2) \leq C_{\lambda}
    \end{equation*}
    for all $N \in \N$, which directly implies $\Vert \XstatOptbar(k) - \XstatOpt(k) \Vert_{L^2} \to 0$ for $k \to \infty$.
  \end{proof}

\section{Numerical examples} \label{sec:Examples}
  
  In this section, we illustrate the theoretical results of this paper. To do this, we consider the stochastic OCP
  \begin{equation} \label{example}
      \begin{split}
          \min_{U\args} &\sum_{k=0}^{N-1} \Exp{X_1(k)^2 + 5X_2(k)^2 + U(k)^2 + X_1(k) - 0.5 U(k)} + \gamma \V[X_1(k)] \\
          &s.t. ~ X(k+1) = 
            \left( \begin{array}{cc}
            1.12 & 0\phantom{.00} \\
            0.26 & 0.88 \end{array} \right) X(k)
            + 
            \left( \begin{array}{r}
            0.05 \\
            -0.05 \end{array} \right) U(k) + W(k)
      \end{split}
  \end{equation}
  with $\gamma = 5$ and initial condition $X(0) = X_0 \sim \Normal\left((0.5,0.8)^T, \mbox{diag}(0.5^2,0.8^2)\right)$, which is a modified version of \cite[Section 3.3]{Ou2021}. 

  We consider two settings: Gaussian distributed  and gamma distributed noise.
  We first consider the case that $W(k) \sim \Normal\left((0.2,0.2)^T, \mbox{diag}(0.03,0.03)\right)$ for all $k=0,\ldots,N-1$. Then the corresponding stationary pair from Theorem~\ref{thm:StochDissi} is given by
  \begin{align*}
      \XstatOpt(k+1) &= \left( \begin{array}{cc}
            1.12 & 0\phantom{.00} \\
            0.26 & 0.88 \end{array} \right) \XstatOpt(k)
            + 
            \left( \begin{array}{r}
            \phantom{-}0.05 \\
            -0.05 \end{array} \right) \UstatOpt(k) + W(k) \\
        \UstatOpt(k) &= K \left( \XstatOpt(k) - (-1.116, -0.199)^T \right) - 1.323
  \end{align*}
  with $K = -(7.679, -0.388)$ and initial distribution $\XstatOpt(0) \sim \Normal(\mu^s, \Sigma^s)$ where
  \begin{equation} \label{example:stationaryMoments}
      \mu^s = (-1.116, -0.199)^T, ~\mbox{and}~ \Sigma^s = \left( \begin{array}{rr}
            0.063 & 0.054\\
            0.054 &  0.619 \end{array} \right)
  \end{equation}
  (all numbers are rounded to four digits). 
  Here we have obtained $K$ by solving the Riccati equation \eqref{eq:algRiccati} and the expectation $\mu_s$ as well as the linear shift in $\UstatOpt(k)$ are determined as explained in Remark~\ref{rem:StationaryExpectation}. The covariance $\Sigma^s$ is then computed by solving the Lyapunov-equation $\Sigma^s = (A+BK) \Sigma^s (A+BK)^T + \Sigma_W$, cf.\ \cite[Lemma~3]{CDCPaper}. Note that in general this approach is not sufficient to fully characterize the stationary distribution, but since here we are considering a purely Gaussian setting, the stationary distribution must be Gaussian again due to the linearity of the system and is thus fully characterized by its first two moments.
  To visualize the different turnpike behaviors defined in Section~\ref{sec:Turnpike}, we solve OCP~\eqref{example} over different optimization horizons $N=20,30,\ldots,80$. To this end, we use Polynomial Chaos Expansions (PCE) and its \emph{Julia} implementation \emph{PolyChaos.jl} \cite{Muehlpfordt2020} for obtaining the solutions numerically. For an introduction to PCE we refer to \cite{Sullivan2015} and for more details on its usage for our simulations to \cite{Ou2021}. Figure~\ref{fig:CSTRGaussianDiffExp} shows the mean-square distance between the stationary pair $(\XstatOpt\args,\UstatOpt\args)$ and the solutions of OCP~\eqref{example} where we can clearly observe the turnpike property from Theorem~\ref{thm:TurnpikeL2}. Here, we approximated the mean-square distance via Monte-Carlo simulations
  \begin{equation*}
    \Exp{\Vert X_U(k,X_0) - \XstatOpt(k) \Vert^2} \approx \frac{1}{n} \sum_{j=1}^n \Vert X_U(k,X_0,w_j) - \XstatOpt(k,w_j) \Vert^2
  \end{equation*}
  and
  \begin{equation*}
       \Exp{\Vert U(k) - \UstatOpt(k) \Vert^2} \approx \frac{1}{n} \sum_{j=1}^n \Vert U(k,w_j) - \UstatOpt(k,w_j) \Vert^2
  \end{equation*}
  where the PCE solutions were sampled for $n=10^4$ noise realizations $w_j(0),\ldots$, $w_j(N-1)$, $j=0,\ldots,n$ and with stochastically independent initial values $X_0$ and ${\XstatOpt}_0$.

\begin{figure}[ht]
    \begin{center}
        \includegraphics[width=1\textwidth]{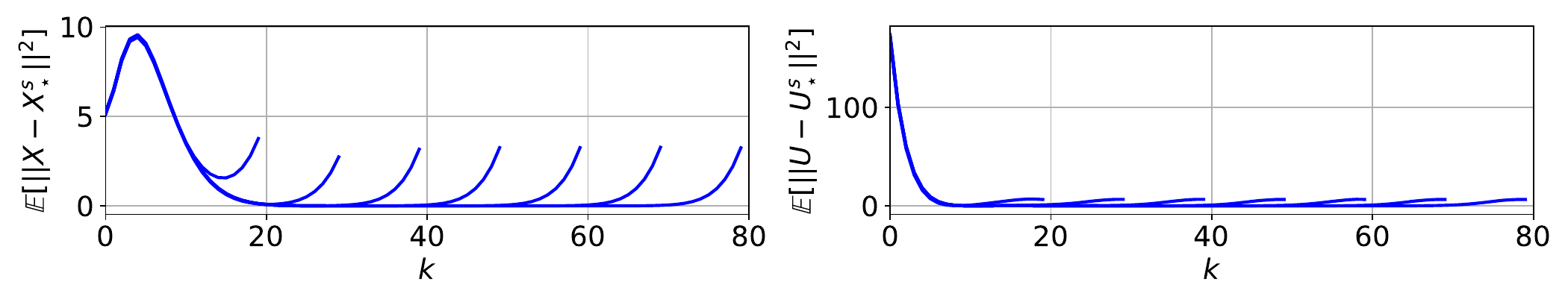}
        \vspace{-0.8cm}
        \caption{Mean-square distance between the stationary pair and the optimal solutions on different horizons. \label{fig:CSTRGaussianDiffExp}} 		
    \end{center}
\end{figure}

    To illustrate the pathwise turnpike behavior of the solutions, we consider two fixed but different realization paths of $W(\cdot)$ given by $w_{i} = (w_{i}(0), \ldots, w_{i}(N-1))$ for $i=1,2$. We simulate the states and controls according to equation \eqref{example} and the optimal stationary pair with $W(k) = w_i(k)$ and with random initial values following the desired distributions. Figure~\ref{fig:CSTRGaussianSample2Rows} shows the results for the  fixed noise-sequences $w_{i}$ in each row and the resulting turnpike behavior of the realization. Notice that Figure~\ref{fig:CSTRGaussianSample2Rows} is not a direct visualization of the pathwise turnpike in probability from Theorem~\ref{thm:TurnpikeProb} since a single realization is a null set in our probabilistic sense. However, it illustrates the pathwise behavior and, thus, it supports Theorem~\ref{thm:TurnpikeProb} being valid.

\begin{figure}[ht]
    \begin{center}
        \includegraphics[width=1\textwidth]{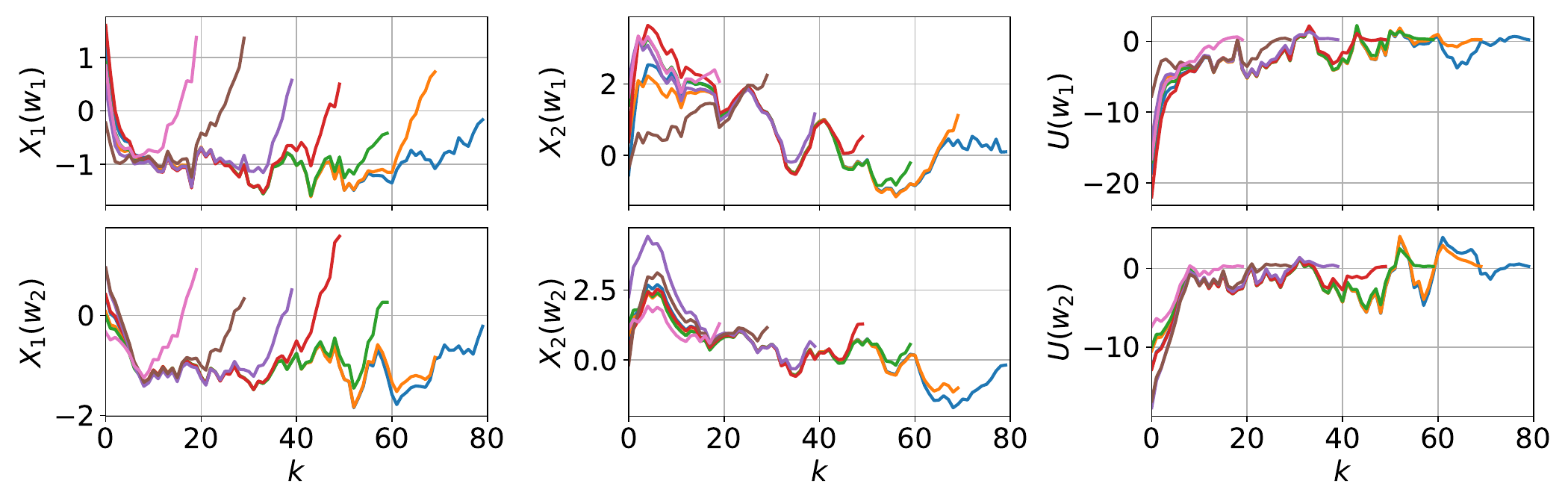}
        \vspace{-0.8cm}	
        \caption{State and control trajectories for different initial conditions and horizons with one identical Gaussian noise realization in each row. \label{fig:CSTRGaussianSample2Rows}}
    \end{center}
\end{figure}

    Since for Gaussian distributions the exact calculation of the Wasserstein distance of order $2$ is possible, we also plot this distance between the solutions of the OCP~\eqref{example} and the stationary pair $(\XstatOpt\args,\UstatOpt\args)$. The results are shown in Figure~\ref{fig:CSTRGaussianWasserstein}. Again we can observe that the system exhibits the turnpike property, here in distribution as defined in Theorem~\ref{thm:TurnpikeDistr}.

\begin{figure}[ht]
    \begin{center}
        \includegraphics[width=1\textwidth]{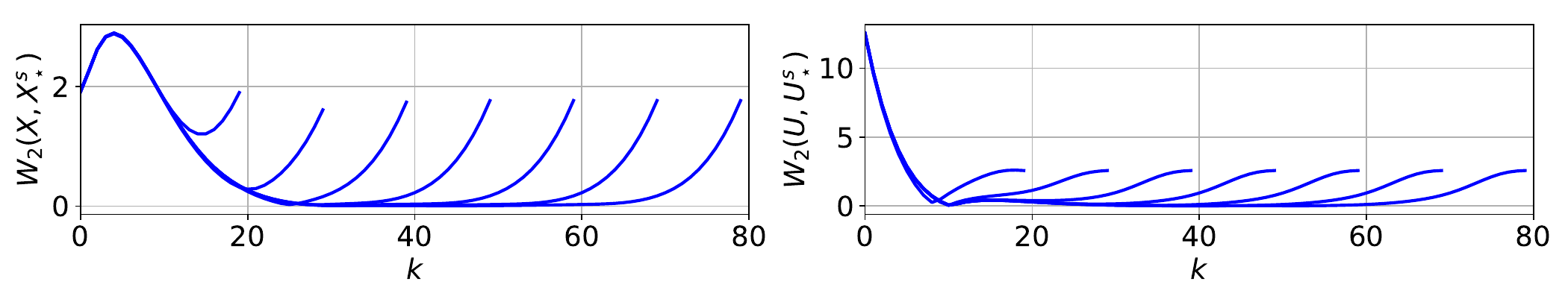}
        \vspace{-0.8cm}
        \caption{Wasserstein distance of order $2$ between the stationary distribution and the distribution of the optimal trajectories on different horizons $N$. \label{fig:CSTRGaussianWasserstein}}	
    \end{center}
\end{figure}

    As we have shown in Theorem~\ref{thm:TurnpikeMoments}, the turnpike behavior in distribution with respect to the Wasserstein distance allows us to make additional statements concerning the first two moments of the solutions. Figure~\ref{fig:CSTRGaussianMoments} depicts the evolution of the expectation and the variance of the solutions for different horizon lengths together with the corresponding moment of the stationary distribution. Once more, we can observe that the trajectories spend most of their time near the stationary solution, which visualizes the turnpike property from Theorem~\ref{thm:TurnpikeMoments}.

\begin{figure}[ht]
    \begin{center}
        \includegraphics[width=1\textwidth]{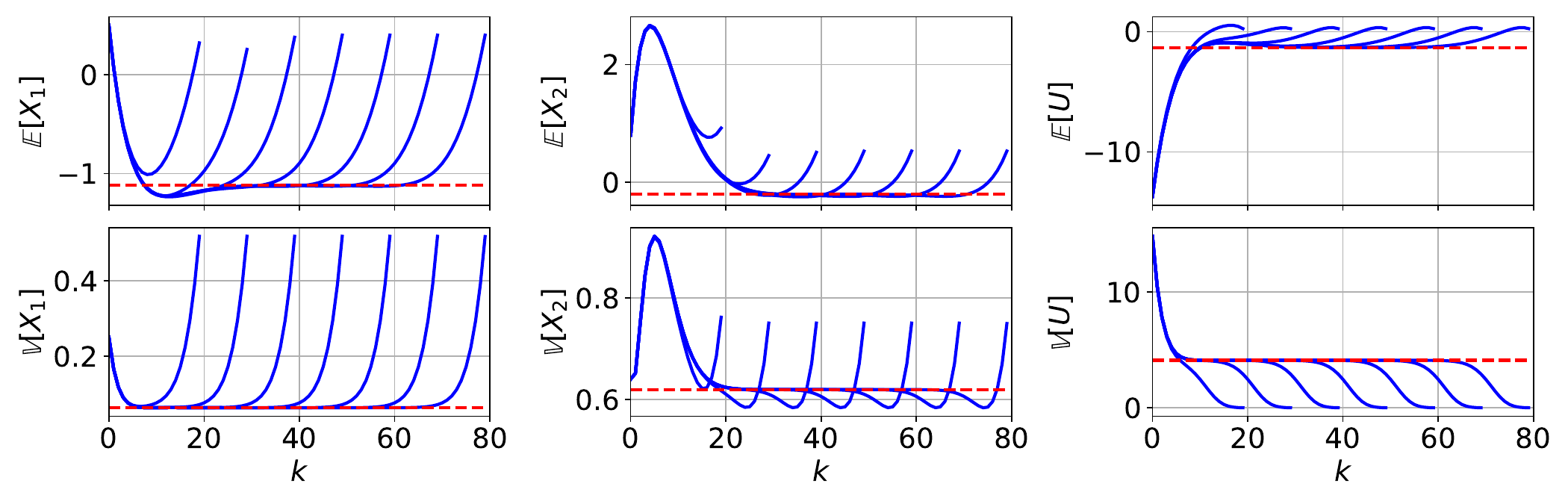}
        \vspace{-0.8cm}
        \caption{Evolution of the expectations and variances for different horizons $N$ (solid blue) together with the corresponding moment of the stationary distribution (dashed red). \label{fig:CSTRGaussianMoments}}
    \end{center}
\end{figure}

    Next we turn to the case of non-Gaussian disturbances to demonstrate that our theoretical results remain applicable. Specifically, we consider the stochastic OCP \eqref{example} but now with 
    $W(k)=(W_1(k),W_2(k))^T$, $W_1(k),W_2(k) \sim \Gamma(1.33,0.15)$
    and $\covar(W_1(k),W_2(k)) = 0$ for $k=0,\ldots,N-1$. 
    First, we plot solutions for two fixed realizations of the gamma-distributed noise, see Figure~\ref{fig:CSTRGammaSample2Rows}. 
    Once again, we observe the pathwise turnpike property analogous to Figure~\ref{fig:CSTRGaussianSample2Rows}, which illustrates that the turnpike property from Theorem~\ref{thm:TurnpikeProb} also holds for non-Gaussian distributions.
    Unfortunately, in this example the numerical validation of the mean-square turnpike and the distributional behavior is quite challenging.
    This is due to the difficulty of computing the distributions of the stationary pair, since these distributions are now not fully characterized by their first two moments, as in the Gaussian setting of the previous example. Thus, we would have to solve the problem \eqref{eq:StationaryOptimality} explicitly, which is in general very hard, since we have to perform a minimization over the infinite dimensional spaces of probability measures $\mathcal{M}$ and feedbacks $\mathbb{F}(\R^n,\R^l)$. 
    Hence, we do not attempt to plot the counterparts of  quantities shown in Figure~\ref{fig:CSTRGaussianDiffExp} and \ref{fig:CSTRGaussianWasserstein} in the non-Gaussian setting. However, via PCE we can calculate the first two moments of the stationary distribution.
    Note that these stationary moments are the same as in equation \eqref{example:stationaryMoments}, since our gamma-distributed noise has the same expectation and variance as the Gaussian noise before, and the evolution of the moments does only depend on the moments of the noise and not on its exact distribution. Thus, we get the same results as shown in Figure~\ref{fig:CSTRGaussianMoments}.

    \begin{figure}[ht]
        \begin{center}
            \includegraphics[width=1\textwidth]{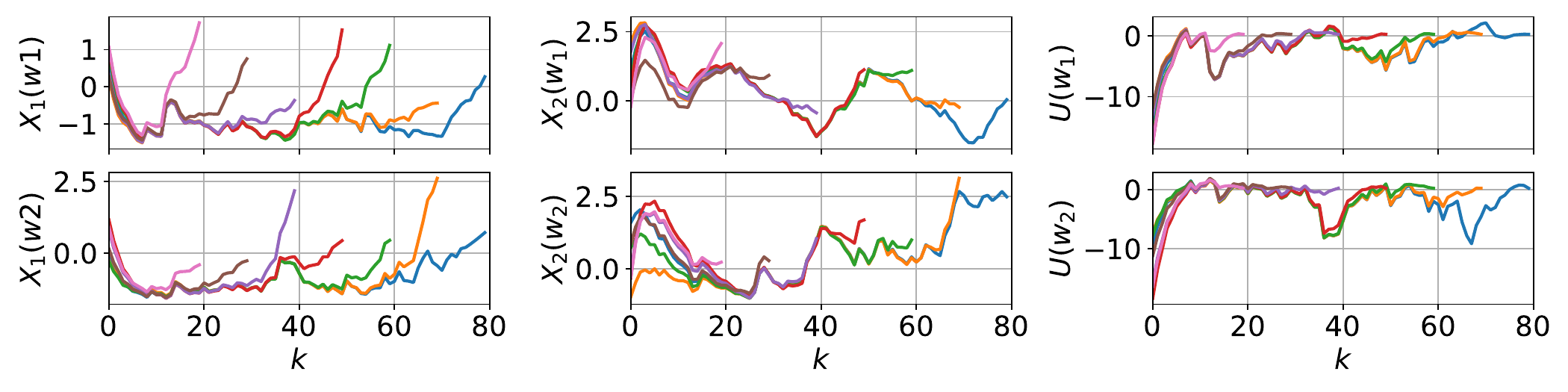}
            \vspace{-0.8cm}
            \caption{State and control trajectories for different initial conditions and horizons with one identical gamma-distributed noise realization in each row.}
            \label{fig:CSTRGammaSample2Rows}
        \end{center}
    \end{figure}

    In order to give an alternative way of illustrating the distributional turnpike property without knowledge of the stationary distribution, Figure~\ref{fig:CSTRGammaDist} illustrates the evolution of the probability density functions (PDFs) defining the marginal distributions of $X_1$ and $X_2$ for the horizon $N=80$. We can observe that the PDFs are almost constant in the middle of the time horizon which is a clear indicator of a turnpike property in distribution. Moreover, observe the asymmetric nature of the PDF for $X_1$.

    \begin{figure}[ht]
    \centering
        \begin{minipage}[t]{.45\textwidth}
            \centering
            \includegraphics[width=\textwidth,trim={30mm 28mm 9mm 38mm},clip]{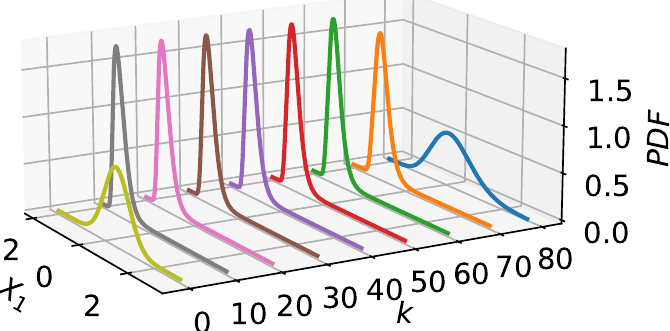}
        \end{minipage}
        \begin{minipage}[t]{.45\textwidth}
            \centering
            \includegraphics[width=\textwidth,trim={30mm 28mm 9mm 38mm},clip]{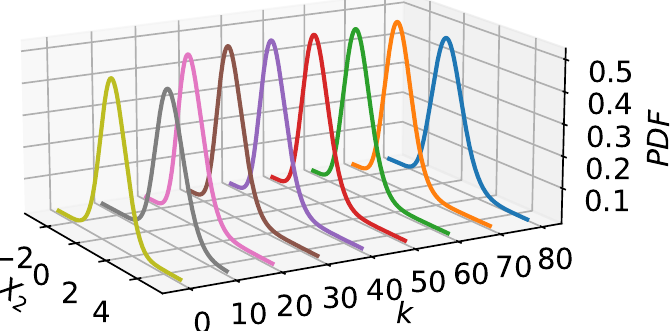}
        \end{minipage}  
        \label{fig:CSTRGammaDist}
        \caption{Evolution of the PDFs defining the marginal distributions of the states on time horizon $N=80$ for gamma-distributed noise.}
    \end{figure}

    Additionally, Figure~\ref{fig:DistComparison} compares the PDFs in the Gaussian and gamma-distributed setting for the solutions on horizon $N=80$ at time $k=40$. Given that this is the middle of the time horizon, it can be reasonably assumed that these PDFs are a good approximation of the stationary distributions according to the definition of the turnpike property. As we can see, in the gamma-distributed setting the stationary distribution is asymmetric and differs significantly from the Gaussian one. Thus, we can observe that the explicit behavior of the solutions depends on the specific distribution of the noise although the proofs in Section~\ref{sec:Dissi} and Section~\ref{sec:Turnpike} do not take them directly into account.

\begin{figure}[ht]
    \begin{center}
        \includegraphics[width=1\textwidth]{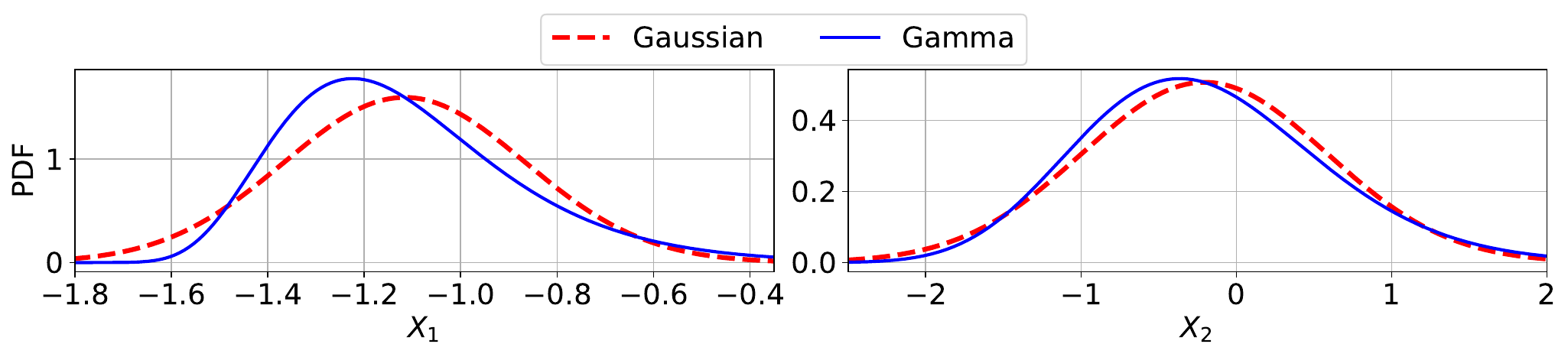}
        \vspace{-0.8cm}
        \caption{Comparison of the PDFs defining the marginal distributions of the states at $k=40$ for Gaussian and gamma-distributed noise.
        \label{fig:DistComparison}}
    \end{center}
\end{figure}

\section{Conclusion and outlook} \label{sec:Conclusion}

In this paper, we have proven that the mean-square dissipativity property holds for a generalized version of the stochastic linear-quadratic optimal control problem by explicitly constructing a storage function for them. Further, we have shown that a pair of stochastic stationary processes replaces the deterministic steady state in our time-varying dissipativity notion and that a stationary optimization problem characterizes the distribution of this pair. Moreover, we have shown that we can conclude several turnpike properties from our dissipativity notion. These turnpike properties include the mean-square turnpike behavior, the pathwise turnpike in probability, the turnpike in distribution with respect to the Wasserstein distance, and the turnpike of the expectation and variance. These types of turnpike behaviors were also illustrated by numerical simulations for a two-dimensional system.

Since \cite{Sun2022,Sun2023} show that the turnpike property also holds for multiplicative noise, a question for future research would be if we could obtain similar discrete-time results using dissipativity-based approaches.
Moreover, it would be interesting to investigate if we could show that the mean-square dissipativity also holds for more general problems, which consider, for example, economic costs, nonlinear dynamics, or constraints. 
Developing efficient methods to approximate the stationary distribution would also be important since this can already be challenging in the linear-quadratic setting as described in Section~\ref{sec:Examples}. 
For deterministic optimal control problems, it is known that under suitable technical assumptions the turnpike property and strict dissipativity are actually equivalent \cite{Gruene2016}. It would be very interesting to investigate whether this is also true for the stochastic setting in the $L^2$ sense. Further, the usage of our results in the context of model predictive control should be investigated. 
Finally, although Remark~\ref{rem:StationarityGrosZanon} gives a first hint on the link between our dissipativity notion and the one presented in \cite{Gros2022}, the exact connection remains an open question. 

\section*{Acknowledgment} We thank the anonymous reviewers for their valuable and constructive comments for improving the paper.

\bibliographystyle{siamplain}
\bibliography{references}

\begin{thebibliography}{10}

\bibitem{Caines2018}
{\sc P.~Caines}, {\em Linear stochastic systems}, John Wiley \& Sons, Inc., 1987.

\bibitem{Damm2014}
{\sc T.~Damm, L.~Gr\"{u}ne, M.~Stieler, and K.~Worthmann}, {\em An exponential turnpike theorem for dissipative discrete time optimal control problems}, {SIAM} Journal on Control and Optimization, 52 (2014), pp.~1935--1957, \url{https://doi.org/10.1137/120888934}.

\bibitem{DasGupta2008}
{\sc A.~DasGupta}, {\em Asymptotic theory of statistics and probability}, vol.~180, Springer, 2008.

\bibitem{Dorfman1958}
{\sc R.~Dorfman, P.~A. Samuelson, and R.~M. Solow}, {\em Linear programming and economic analysis}, A Rand Corporation Research Study, McGraw-Hill, New York-Toronto-London, 1958.

\bibitem{Farina2015}
{\sc M.~Farina, L.~Giulioni, L.~Magni, and R.~Scattolini}, {\em An approach to output-feedback mpc of stochastic linear discrete-time systems}, Automatica, 55 (2015), pp.~140--149.

\bibitem{Faulwasser2018}
{\sc T.~Faulwasser, L.~Gr\"une, and M.~A. M\"uller}, {\em Economic nonlinear model predictive control}, Foundations and Trends\textsuperscript{\textregistered} in Systems and Control, 5 (2018), pp.~1--98.

\bibitem{Faulwasser2022}
{\sc T.~Faulwasser and L.~Grüne}, {\em Turnpike properties in optimal control}, in Numerical Control: Part A, Elsevier, 2022, pp.~367--400, \url{https://doi.org/10.1016/bs.hna.2021.12.011}.

\bibitem{Faulwasser2017}
{\sc T.~Faulwasser, M.~Korda, C.~N. Jones, and D.~Bonvin}, {\em On turnpike and dissipativity properties of continuous-time optimal control problems}, Automatica, 81 (2017), pp.~297--304, \url{https://doi.org/10.1016/j.automatica.2017.03.012}.

\bibitem{Fristedt1997}
{\sc B.~Fristedt and L.~Gray}, {\em A Modern Approach to Probability Theory}, Birkhäuser Boston, 1997, \url{https://doi.org/10.1007/978-1-4899-2837-5}.

\bibitem{Gros2022}
{\sc S.~Gros and M.~Zanon}, {\em Economic {MPC} of {M}arkov decision processes: Dissipativity in undiscounted infinite-horizon optimal control}, Automatica, 146 (2022), p.~110602, \url{https://doi.org/10.1016/j.automatica.2022.110602}.

\bibitem{Gruene2022}
{\sc L.~Gr\"{u}ne}, {\em Dissipativity and optimal control: examining the turnpike phenomenon}, IEEE Control Syst. Mag., 42 (2022), pp.~74--87, \url{https://doi.org/10.1109/tac.1971.1099831}, \url{https://doi.org/10.1109/tac.1971.1099831}.

\bibitem{Gruene2013}
{\sc L.~Grüne}, {\em Economic receding horizon control without terminal constraints}, Automatica, 49 (2013), pp.~725--734, \url{https://doi.org/10.1016/j.automatica.2012.12.003}.

\bibitem{Gruene2018}
{\sc L.~Grüne and R.~Guglielmi}, {\em Turnpike properties and strict dissipativity for discrete time linear quadratic optimal control problems}, {SIAM} Journal on Control and Optimization, 56 (2018), pp.~1282--1302, \url{https://doi.org/10.1137/17m112350x}.

\bibitem{Gruene2016}
{\sc L.~Grüne and M.~A. Müller}, {\em On the relation between strict dissipativity and turnpike properties}, Systems {\&} Control Letters, 90 (2016), pp.~45--53, \url{https://doi.org/10.1016/j.sysconle.2016.01.003}.

\bibitem{Gruene2018a}
{\sc L.~Grüne, S.~Pirkelmann, and M.~Stieler}, {\em Strict dissipativity implies turnpike behavior for time-varying discrete time optimal control problems}, in Lecture Notes in Economics and Mathematical Systems, Springer International Publishing, 2018, pp.~195--218, \url{https://doi.org/10.1007/978-3-319-75169-6_10}.

\bibitem{Heirung2018}
{\sc T.~A.~N. Heirung, J.~A. Paulson, J.~O’Leary, and A.~Mesbah}, {\em Stochastic model predictive control—how does it work?}, Computers \& Chemical Engineering, 114 (2018), pp.~158--170.

\bibitem{Hewing2020}
{\sc L.~Hewing, K.~P. Wabersich, and M.~N. Zeilinger}, {\em Recursively feasible stochastic model predictive control using indirect feedback}, Automatica, 119 (2020), p.~109095.

\bibitem{Kolokoltsov2012}
{\sc V.~Kolokoltsov and W.~Yang}, {\em Turnpike theorems for {M}arkov games}, Dynamic Games and Applications, 2 (2012), pp.~294--312, \url{https://doi.org/10.1007/s13235-012-0047-6}.

\bibitem{Kumar2015}
{\sc P.~Kumar and P.~Varaiya}, {\em Stochastic systems: estimation, identification and adaptive control}, Prentice-Hall, Inc., 1986.

\bibitem{Marimon1989}
{\sc R.~Marimon}, {\em Stochastic turnpike property and stationary equilibrium}, Journal of Economic Theory, 47 (1989), pp.~282--306, \url{https://doi.org/10.1016/0022-0531(89)90021-5}.

\bibitem{Mesbah2016}
{\sc A.~Mesbah}, {\em Stochastic model predictive control: An overview and perspectives for future research}, IEEE Control Systems Magazine, 36 (2016), pp.~30--44.

\bibitem{Meyn1993}
{\sc S.~P. Meyn and R.~L. Tweedie}, {\em {M}arkov Chains and Stochastic Stability}, Springer London, 1993, \url{https://doi.org/10.1007/978-1-4471-3267-7}.

\bibitem{Muehlpfordt2020}
{\sc T.~M{\"u}hlpfordt, F.~Zahn, V.~Hagenmeyer, and T.~Faulwasser}, {\em Polychaos. jl—a julia package for polynomial chaos in systems and control}, IFAC-PapersOnLine, 53 (2020), pp.~7210--7216.

\bibitem{Mueller2021}
{\sc M.~A. M{\"u}ller}, {\em Dissipativity in economic model predictive control: beyond steady-state optimality}, Recent advances in model predictive control: theory, algorithms, and applications,  (2021), pp.~27--43.

\bibitem{Ou2021}
{\sc R.~Ou, M.~H. Baumann, L.~Grüne, and T.~Faulwasser}, {\em A simulation study on turnpikes in stochastic {LQ} optimal control}, {IFAC}-{PapersOnLine}, 54 (2021), pp.~516--521, \url{https://doi.org/10.1016/j.ifacol.2021.08.294}.

\bibitem{Protter2005}
{\sc P.~E. Protter}, {\em Stochastic Integration and Differential Equations}, Springer Berlin Heidelberg, 2005, \url{https://doi.org/10.1007/978-3-662-10061-5}.

\bibitem{Ramsey1928}
{\sc F.~P. Ramsey}, {\em A mathematical theory of saving}, The Economic Journal, 38 (1928), p.~543, \url{https://doi.org/10.2307/2224098}.

\bibitem{Aastroem1970}
{\sc K.~J. \r{A}ström}, {\em Introduction to Stochastic Control Theory (Mathematics in Science and Engineering, Volume 70)}, Academic Press, 1970.

\bibitem{Risbeck2020}
{\sc M.~J. Risbeck and J.~B. Rawlings}, {\em Economic model predictive control for time-varying cost and peak demand charge optimization}, IEEE Transactions on Automatic Control, 65 (2020), pp.~2957--2968, \url{https://doi.org/10.1109/TAC.2019.2939633}.

\bibitem{CDCPaper}
{\sc J.~Schießl, R.~Ou, T.~Faulwasser, M.~H. Baumann, and L.~Grüne}, {\em Pathwise turnpike and dissipativity results for discrete-time stochastic linear-quadratic optimal control problems}, in 2023 62nd IEEE Conference on Decision and Control (CDC), 2023, pp.~2790--2795, \url{https://doi.org/10.1109/CDC49753.2023.10384081}.

\bibitem{Sullivan2015}
{\sc T.~J. Sullivan}, {\em Introduction to uncertainty quantification}, vol.~63, Springer, 2015.

\bibitem{Sun2022}
{\sc J.~Sun, H.~Wang, and J.~Yong}, {\em Turnpike properties for stochastic linear-quadratic optimal control problems}, Chinese Annals of Mathematics, Series B, 43 (2022), pp.~999--1022, \url{https://doi.org/10.1007/s11401-022-0374-x}.

\bibitem{Sun2023}
{\sc J.~Sun and H.~Wu}, {\em Long-time behavior of stochastic linear-quadratic optimal control problems}, in 2023 62nd IEEE Conference on Decision and Control (CDC), 2023, pp.~2796--2802, \url{https://doi.org/10.1109/CDC49753.2023.10383586}.

\bibitem{Villani2009}
{\sc C.~Villani}, {\em The wasserstein distances}, Optimal Transport: Old and New,  (2009), pp.~93--111.

\bibitem{Neumann1945}
{\sc J.~von Neumann}, {\em A model of general economic equilibrium}, The Review of Economic Studies, 13 (1945), p.~1, \url{https://doi.org/10.2307/2296111}.

\bibitem{willems1972a}
{\sc J.~C. Willems}, {\em Dissipative dynamical systems part {I}: General theory}, Archive for rational mechanics and analysis, 45 (1972), pp.~321--351.

\bibitem{willems1972b}
{\sc J.~C. Willems}, {\em Dissipative dynamical systems part {II}: Linear systems with quadratic supply rates}, Archive for rational mechanics and analysis, 45 (1972), pp.~352--393.

\end{thebibliography}
\end{document}